\documentclass[a4paper,11pt]{amsart}
\usepackage{amssymb}
\usepackage{amsmath}
\usepackage{latexsym}
\usepackage{eepic}
\usepackage{epsfig}
\usepackage{graphicx}

\usepackage{eufrak}
\usepackage{mathrsfs}
\usepackage[english]{babel}
\usepackage{color}
\usepackage[all,cmtip]{xy}
\usepackage{verbatim}

\DeclareMathAlphabet{\mathpzc}{OT1}{pzc}{m}{it}
\newtheorem{theorem}[equation]{Theorem}
\newtheorem{thm}[equation]{Theorem}
\newtheorem*{theorem*}{Theorem}
\newtheorem{theorem-definition}[equation]{Theorem-Definition}
\newtheorem{lemma-definition}[equation]{Lemma-Definition}
\newtheorem{definition-prop}[equation]{Proposition-Definition}
\newtheorem{corollary}[equation]{Corollary}
\newtheorem{prop}[equation]{Proposition}
\newtheorem*{prop*}{Proposition}
\newtheorem{lemma}[equation]{Lemma}
\newtheorem{cor}[equation]{Corollary}
\newtheorem{definition}[equation]{Definition}
\newtheorem*{definition*}{Definition}

\newtheorem*{conjecture*}{Conjecture}

\newtheorem*{theoremtag}{Theorem \ref{thm:nonrat}}

\theoremstyle{definition}
\newtheorem{exam}[equation]{Example}
\newtheorem{example}[equation]{Example}
\newtheorem{examples}[equation]{Examples}

\newtheorem*{question*}{Question}
\newtheorem{remark}[equation]{Remark}

\newcommand{\LL}{\ensuremath{\mathbb{L}}}

\newcommand{\N}{\ensuremath{\mathbb{N}}}
\newcommand{\Z}{\ensuremath{\mathbb{Z}}}
\newcommand{\Q}{\ensuremath{\mathbb{Q}}}

\newcommand{\R}{\ensuremath{\mathbb{R}}}

\newcommand{\OO}{\widehat{\mathscr{O}}}

\newcommand{\A}{\ensuremath{\mathbb{A}}}

\newcommand{\cO}{\ensuremath{\mathscr{O}}}
\newcommand{\cX}{\ensuremath{\mathscr{X}}}

\newcommand{\cZ}{\ensuremath{\mathscr{Z}}}

\newcommand{\cU}{\ensuremath{\mathscr{U}}}

\newcommand{\cY}{\ensuremath{\mathscr{Y}}}
\newcommand{\cW}{\ensuremath{\mathscr{W}}}

\newcommand{\cD}{\ensuremath{\mathscr{D}}}
\newcommand{\cE}{\ensuremath{\mathscr{E}}}

\renewcommand{\R}{\ensuremath{\mathbb{R}}}

\renewcommand{\A}{\ensuremath{\mathbb{A}}}

\renewcommand{\cU}{\ensuremath{\mathscr{U}}}

\renewcommand{\cZ}{\ensuremath{\mathscr{Z}}}
\renewcommand{\cY}{\ensuremath{\mathscr{Y}}}

\newcommand{\Spec}{\ensuremath{\mathrm{Spec}\,}}

\newcommand{\red}{\mathrm{MR}}

\newcommand{\Gal}{\mathrm{Gal}}
\newcommand{\Hom}{\mathrm{Hom}}

\newcommand{\llbr}{[\negthinspace[}
\newcommand{\rrbr}{]\negthinspace]}
\newcommand{\llpar}{(\negthinspace(}
\newcommand{\rrpar}{)\negthinspace)}

\newcommand{\Var}{\mathrm{Var}}
\newcommand{\VF}{\mathrm{VF}}
\newcommand{\Vol}{\mathrm{Vol}}

\newcommand{\gal}{\widehat{\mu}}
\newcommand{\gro}{\mathbf{K}}
\newcommand{\Res}{\mathrm{Res}}

\newcommand{\loga}[1]{\mathscr{#1}^{\dagger}}
\newcommand{\pri}{R(\infty)} 
\newcommand{\pfi}{K(\infty)} 

\def\L{\mathbb L}
\def\P{\mathbb P}
\def\wt{\widetilde}

\numberwithin{equation}{subsection}

\title[Motivic nearby fiber and stable rationality]{The motivic nearby fiber and degeneration of stable rationality}
\author[Johannes Nicaise]{Johannes Nicaise}
\address{Imperial College,
Department of Mathematics, South Kensington Campus,
London SW72AZ, UK, and KU Leuven, Department of Mathematics, Celestijnenlaan 200B, 3001 Heverlee, Belgium} \email{j.nicaise@imperial.ac.uk}
\thanks{Johannes Nicaise is supported by the ERC Starting Grant MOTZETA (project 306610) of the European Research Council, and  by long term structural funding (Methusalem
grant) of the Flemish Government.}

\author{Evgeny Shinder}
\address{School of Mathematics and Statistics, University of Sheffield,
Hounsfield Road, S3 7RH, UK, and
National Research University Higher School of Economics, Russian Federation}
\email{e.shinder@sheffield.ac.uk}
\thanks{Evgeny Shinder is partially supported by Laboratory of Mirror Symmetry NRU HSE, RF government grant, ag. N 14.641.31.0001.}

\begin{document}
\begin{abstract}
We prove that stable rationality specializes in regular families whose fibers are integral and have at most ordinary double points as singularities. Our proof is based on motivic specialization techniques and the criterion of Larsen and Lunts for stable rationality in the Grothendieck ring of varieties.
\end{abstract}

\maketitle

\section{Introduction}
Let $k$ be a field of characteristic zero.
An $n$-dimensional $k$-variety $X$ is called rational if $X$ is birational to the projective space $\P^n$, and
stably rational if $X \times \P^m$ is rational for some $m \ge 0$.
It is a natural question, considered recently in particular in \cite{dFF, Voisin, Totaro, Perry},
how rationality and related notions behave in families.

From our perspective the most natural question is that of specialization: if a very general member of a flat family $\cX \to S$ of varieties
has a certain property, does every member of the family have the same property?
Degenerating smooth varieties to cones over singular varieties shows that rationality and stably rationality do not specialize even for terminal singularities \cite{dFF, Totaro, Perry},
thus these questions are meaningful only for smooth families or for some very restricted classes of singularities.

It is known that properties such as ruledness, uniruledness and rational connectedness of varieties specialize in smooth families \cite{Matsusaka, Kollar-book}.
It is also known that rationality specializes in smooth families of $3$-dimensional varieties \cite{Tim, dFF}.

\medskip

In this paper we study specialization property of \emph{stable rationality} in arbitrary dimension.
One of our main result is the following:

\begin{theoremtag}
Let $k$ be a field of characteristic zero. Let $f\colon \cX \to C$ be a proper flat morphism with $\cX$, $C$ connected smooth $k$-schemes and $\dim(C) = 1$.
If the geometric generic fiber of $f$ is stably rational, then all geometric fibers with at most ordinary double points as singularities have a stably rational irreducible component.
\end{theoremtag}

For a precise definition of what we mean by ordinary double points, see Definition \ref{def:ordinary}; they are not necessarily isolated singularities.

\medskip

One consequence of the theorem is that, if $\cX \to S$ is a smooth proper morphism of $k$-varieties, then the set $S_{sr}$ of points in $S$
parameterizing stably rational geometric fibers is a countable union of Zariski closed subsets; see Corollary \ref{cor:stabratlocus} for a more general statement.
 On the other hand, it is known that rationality and stable rationality are not open properties: there exist examples where $S$ is connected and
 $S_{sr}$ is a non-empty strict subset of $S$ \cite{HPT16}.

The theorem also explains why in practice, for a given smooth family of varieties, proving that one member is stably irrational is as hard as proving
that a very general member is stably irrational. This phenomenon can be illustrated on smooth cubic threefolds: they are known to be irrational \cite{CG}
but stable rationality is known neither for very general nor for specific cubics.

\medskip

Degeneration techniques are known to be useful in proving irrationality and stable irrationality. Beauville used a degeneration argument for the intermediate Jacobian
while proving irrationality for Fano threefolds \cite{Beauville-jac}.
Koll\'ar used degeneration to characteristic $p$ in his proof of non-ruledness (and hence, irrationality) of hypersurfaces in $\P^n$ of high degree \cite{Kollar-hypersurfaces}.

Our result on specialization of stable rationality, as well as its proof, have been inspired by the corresponding result on specialization for the universal Chow zero triviality introduced by
Voisin \cite[Theorem 1.1]{Voisin}. The latter specialization result has been
used to solve some long-standing questions about stable irrationality of certain very general cyclic coverings, high degree hypersurfaces in projective spaces
and conic bundles \cite{Voisin, CTP-stab, Beauville-stab, Totaro-stab, HKT-stab}.

When degeneration does not involve characteristic $p$, our approach allows to deduce stable irrationality
directly, without invoking Chow groups; see Theorem \ref{thm:hypers}, Example \ref{ex:qdsolid} for quartic and sextic double solids, where the constructed degeneration has isolated ordinary double points,
and Proposition \ref{thm:quartics} for three-dimensional quartics, where the degeneration has more complicated quadratic singularities.

\medskip

Our proofs of degeneration results for stable rationality rely on the Grothendieck ring of varieties.
 It was shown by Larsen and Lunts that stable rationality of a smooth and proper variety over a field of characteristic zero can be detected on the class of the variety in the Grothendieck ring:
  stable rationality is equivalent with this class being congruent to $1$ modulo the class $\LL$ of the affine line.
 We study two types of specialization maps between Grothendieck rings: Hrushovski and Kazhdan's {\em motivic volume}, which refines the motivic nearby fiber of Denef and Loeser; and the
 {\em motivic reduction}, which is the quotient of the motivic volume by the monodromy action.
 The motivic volume is a $\Z[\L]$-algebra homomorphism whereas the motivic reduction is only a $\Z[\L]$-module homomorphism (see Remark \ref{rem:MRhom} for discussion);
  both morphisms preserve congruences modulo $\L$ and thus can be applied to study stable rationality.

 Explicit formulas for the specialization maps on strict normal crossings degenerations then allow us to establish our specialization results for stable rationality. In order to deal with degenerations to singular fibers, we study singularities with the property that a resolution of singularities does not alter the class in the Grothendieck ring of varieties modulo $\LL$; we call such singularities {\em $\LL$-rational} (Definition \ref{def:Lrat}). Prominent examples are rational surface singularities and, in arbitrary dimension, toric singularities and ordinary double points (Example \ref{exam:Lrat}).

In addition to restricting singularities of the special fiber, in order for our method to apply we typically make an assumption that the total space of the degeneration is $\L$-faithful,
which by definition means that the motivic volume is congruent to the class of the special fiber modulo $\L$. It is not hard to check that nodal degenerations are $\L$-faithful.
For this one constructs their explicit semi-stable model by making a degree two base change followed by a single blow-up of the singular locus of the total space.
We demonstrate how to deal with more complicated singularities in Proposition \ref{prop:quart-model}.

 \medskip

Shortly after the first version of this paper had appeared on the arxiv, Kontsevich and Tschinkel used our method to construct
specialization homomorphisms analogous to the motivic volume and motivic reduction for
a different ring of varieties, which they call the Burnside ring \cite{KT}. This yields a birational version of the motivic nearby fiber.
They used this invariant to prove that rationality and birational type specialize
in smooth and mildly singular families, thus providing an important generalization of our results from
stable rationality to rationality. Even though the main theorem of \cite{KT} is strictly stronger than ours, we believe that our method is still of independent interest.
 Firstly, a typical application of the result in \cite{KT} is to disprove rationality by means of a degeneration argument,
 and this can often be achieved by directly disproving stable rationality, for which our results are sufficient; see for instance the applications in Section \ref{subsec:app}. More importantly, some specific tools are available for the computation of the motivic volume, which do not apply to the birational variant in \cite{KT}. Indeed, it follows from the work of Hrushovski and Kazhdan that the motivic volume is invariant under semi-algebraic bijections, which makes it possible to apply tools from tropical geometry -- see in particular the motivic Fubini theorem in \cite{NiPa}.
   It seems promising to explore the applications of these techniques to rationality questions.

\medskip
We conclude the introduction with a brief overview of the paper. We collect some preliminary definitions and results in Section \ref{sec:prelim}. In Section \ref{sec:vol} we introduce the technology for Grothendieck rings that we will need. In particular, we define the motivic volume and the motivic reduction maps; they are characterized by the properties stated in Theorem \ref{thm:vol} and Proposition \ref{prop:red}, respectively.
We apply these tools to the study of rationality questions in Section \ref{sec:rat}; our main results are Theorem \ref{thm:specbir} and \ref{thm:nonrat}.
We give concrete applications of these results in Section \ref{subsec:app}, see Theorem \ref{thm:hypers}, Example \ref{ex:qdsolid} and Theorem \ref{thm:quartics}.
 Finally, in Appendix \ref{sec:WF}, we give an alternative proof of the existence of the motivic volume and reduction maps, which relies on Weak Factorization instead of motivic integration. We believe that it will make the construction more accessible to algebraic geometers; it also provides a new and useful formula for the motivic volume in terms of log smooth models.


\medskip

The authors thank T.~Bridgeland, I.~Cheltsov, J.-L.~Colliot-Th\'el\`ene, S.~Galkin, J.~Koll\'ar, A.~Kuznetsov, A.~Pirutka, Yu.~Prokhorov, S.~Schreieder, C.~Shramov, and B.~Totaro for discussions, encouragement and e-mail communication.
 The first-named author wishes to thank in particular O.~Wittenberg, with whom he has had several discussions in 2014 on some of the main results in this paper.
  We are grateful to Yu.\,Tschinkel for pointing out an error in a previous version of this paper,
where we mistakenly claimed that the motivic reduction is a ring homomorphism. We would also like to thank the referee for their thoughtful comments and suggestions, which have improved the presentation of the paper. In particular, it was the referee's suggestion to formulate a version of Theorem \ref{thm:nonrat} for reducible fibers.

\section{Preliminaries}\label{sec:prelim}
\subsection{Notation}
Let $k$ be a field of characteristic zero.
 We set $R=k\llbr t\rrbr$ and $K=k\llpar t\rrpar$. For every positive integer $n$, we also write $R(n)=k\llbr t^{1/n}\rrbr$
 and  $K(n)=k\llpar t^{1/n}\rrpar$. Finally, we set
 $$\pri=\bigcup_{n>0} R(n), \quad \pfi=\bigcup_{n>0} K(n).$$
 If $k$ is algebraically closed, then the field $\pfi$ is an algebraic closure of $K$. We write $(\cdot)_k$, $(\cdot)_K$ and $(\cdot)_{\pfi}$ for the base change functors to the categories of $k$-schemes, $K$-schemes and $\pfi$-schemes, respectively.

 For every positive integer $n$, we denote by $\mu_n$ the group scheme of $n$-th roots of unity over $k$.
 We order the positive integers by the divisibility relation and set $\displaystyle \gal=\lim_{\longleftarrow}\mu_n$ where the transition morphism $\mu_{mn}\to \mu_n$ is the $m$-th power map, for all positive integers $m$ and $n$.
 If $k$ contains all roots of unity, then $\mu_n$ and $\gal$ are constant group schemes over $k$, and they are canonically isomorphic to the Galois groups $\Gal(K(n)/K)$ and
 $\Gal(\pfi/K)$, respectively.

\subsection{Constructions on snc-models}
If $X$ is a proper $K$-scheme, then an $R$-model of $X$ is a flat and proper $R$-scheme $\cX$ endowed with an isomorphism $\cX_K\to X$.
 If $X$ is smooth, then we say that $\cX$ is an snc-model for $X$ if $\cX$ is
regular and the special fiber $\cX_k$ is a strict normal crossings divisor (possibly non-reduced).
By Hironaka's resolution of singularities, every $R$-model of $X$ can be dominated by an snc-model. We say that $X$ has {\em semistable reduction} if it has an snc-model with reduced special fiber.

 Let $\cX$ be an snc-model of $X$. We write the special fiber as
$$\cX_k=\sum_{i\in I}N_i E_i$$ where $E_i,\,i\in I$ are the prime components of $\cX_k$ and the $N_i$ are their multiplicities.  For every non-empty subset $J$ of $I$, we set
$$N_J=\gcd\{N_j\,|\,j\in J\}.$$
Moreover, we put
$$E_J=\bigcap_{j\in J}E_j,\qquad E_J^o=E_J\setminus (\bigcup_{i\in I\setminus J}E_i).$$
The set
 $\{E_J^o\,|\,\emptyset\neq J\subset I\}$
is a partition of $\mathscr{X}_k$ into locally closed subsets. We
endow all $E_J$ and $E_J^o$ with their reduced induced subscheme structures.

 Let $J$ be a non-empty subset of $I$, and let
 $$\cX(N_J)\to \cX\times_R R(N_J)$$ be the normalization morphism.
 We denote by $\widetilde{E}^o_J$ the inverse image of $E_J^o$ in $\cX(N_J)_k$. The group scheme $\mu_{N_J}$
 acts on $\cX(N_J)$, and this action turns $\widetilde{E}^o_J$ into a $\mu_{N_J}$-torsor over $E_J^o$ -- see for instance \cite[\S2.3]{Ni-tame}.
  If $m$ is a multiple of $N_J$ and we denote by $\cX(m)$ the normalization of $\cX\times_R R(m)$, then the natural morphism
  $h\colon \cX(m)\to \cX(N_J)$ induces an isomorphism $h^{-1}(\widetilde{E}^o_J)_{\mathrm{red}}\to \widetilde{E}^o_J$, by \cite[3.2.2]{BuNi}. Thus we can use any multiple of $N_J$ to compute the cover $\widetilde{E}_J^o$ of $E_J^o$.
    For every $i\in I$, we write $E_i^o$ and $\widetilde{E}_i^o$ instead
of $E^o_{\{i\}}$ and $\widetilde{E}^o_{\{i\}}$.

\subsection{Grothendieck rings of varieties}
Let $F$ be a field of characteristic zero, and let $G$ be a profinite group scheme over $F$. We say that a quotient group scheme $H$ of $G$ is {\em admissible} if the kernel of
$G(F^a)\to H(F^a)$ is an open subgroup of the profinite group $G(F^a)$, where $F^a$ denotes an algebraic closure of $F$. In particular, $H$ is a finite group scheme over $F$.

 We now define the Grothendieck ring $\gro^{G}(\Var_F)$ of $F$-varieties with $G$-action. As an abelian group, it is characterized by the following presentation:
\begin{itemize}
\item {\em Generators}: isomorphism classes of $F$-schemes $X$ of finite type endowed with a good $G$-action. Here ``good'' means that the action factors through an admissible quotient of $G$ and that we can cover $X$ by $G$-stable affine open subschemes (the latter condition is always satisfied when $X$ is quasi-projective). Isomorphism classes are taken with respect to $G$-equivariant isomorphisms.
\item {\em Relations}: we consider two types of relations.
\begin{enumerate}
\item {\em Scissor relations}: if $X$ is a $F$-scheme of finite type with a good $G$-action and $Y$ is a $G$-stable closed subscheme of $X$, then
$$[X]=[Y]+[X\setminus Y].$$
\item {\em Trivialization of linear actions}: let $X$ be a $F$-scheme of finite type with a good $G$-action, and let $V$ be a $F$-vector scheme of dimension $d$ with a good linear action of $G$.
Then $$[X\times_F V]=[X\times_F \A^d_F]$$ where the $G$-action on $X\times_F V$ is the diagonal action and the action on $\A^d_F$ is trivial.
\end{enumerate}
\end{itemize}
The group $\gro^{G}(\Var_F)$ has a unique ring structure such that $[X]\cdot [X']=[X\times_F X']$ for all $F$-schemes $X$, $X'$ of finite type  with good $G$-action, where the $G$-action
on $X\times_F X'$ is the diagonal action. We write $\LL$ for the class of $\A^1_F$ (with the trivial $G$-action) in the ring $\gro^{G}(\Var_F)$.
 When $G$ is the trivial group, we write $\gro(\Var_F)$ instead of $\gro^{G}(\Var_F)$. When
 $f:\gro^{G}(\Var_F)\to A$ is a function taking values in some set $A$, and $X$ is a
  $F$-scheme of finite type with good $G$-action,  we will usually write $f(X)$ instead of $f([X])$. This applies in particular to the motivic volume and reduction maps that we will construct in Section \ref{sec:vol}.

 If $F'$ is a field extension of $F$, then we have an obvious base change morphism
 $$\gro^{G}(\Var_F)\to \gro^{G}(\Var_{F'}):[X]\mapsto [X\times_F F'].$$
If $G'\to G$ is a continuous morphism of profinite group schemes, then we can also consider the restriction morphism
$$\Res^G_{G'}:\gro^{G}(\Var_F)\to \gro^{G'}(\Var_F).$$
Both of these morphisms are ring homomorphisms.

\subsection{Bittner's theorem and the theorem of Larsen and Lunts}
The structure of the Grothendieck ring of varieties is still quite mysterious, but
 when $F$ is a field of characteristic zero, there are two powerful results that follow from Hironaka's resolution of singularities and the Weak Factorization theorem \cite{WF, Wlod}.
 The first result is due to Bittner \cite[3.1]{bittner} and provides an alternative presentation for the group $\gro(\Var_F)$ that is more convenient for the construction of motivic invariants.

 \begin{theorem}[Bittner]\label{thm:bittner}
Let $F$ be a field of characteristic zero. Then the group $\gro(\Var_F)$ has the following presentation.
\begin{itemize}
\item {\em Generators:} isomorphism classes $[X]$ of smooth and proper $K$-schemes $X$.
\item {\em Relations:} $[\emptyset]=0$, and for every smooth and proper $K$-scheme $X$ and every connected smooth closed subscheme $Y$ of $X$,
$$[\mathrm{Bl}_Y X]-[E]=[X]-[Y]$$
where $\mathrm{Bl}_Y X$ is the blow-up of $X$ along $Y$ and $E$ is the exceptional divisor.
  \end{itemize}
 \end{theorem}

 The second result is due to Larsen and Lunts.
 Let $X$ and $Y$ be reduced $F$-schemes of finite type. Then $X$ and $Y$ are called {\em birational} if they contain isomorphic dense open subschemes (we do not require $X$ and $Y$ to be irreducible).
 They are called {\em stably birational} if $X \times_F \P_F^\ell$ is birational to $Y \times_F \P_F^m$ for some $\ell, m \ge 0$, and $X$ is called {\em stably
rational} if it is stably birational to the point $\Spec F$.
 We denote by $\mathrm{SB}_F$ the set of stable birational equivalence classes of non-empty connected smooth and proper $F$-schemes, and by $\Z[\mathrm{SB}_F]$ the free $\Z$-module on the set
 $\mathrm{SB}_F$.

\begin{theorem}[Larsen-Lunts]\label{thm:LL}
Let $F$ be a field of characteristic zero. There exists a unique group morphism
$$\mathrm{sb}\colon \gro(\Var_F)\to \Z[\mathrm{SB}_F]$$
 that maps the class $[X]$ of each non-empty connected smooth and proper $F$-scheme $X$
 to the stable birational equivalence class of $X$.
This morphism is surjective, and its kernel is the ideal of $\gro(\Var_F)$ generated by $\LL$.
 \end{theorem}

\begin{corollary}\label{cor:LL}
Let $X$ and $X'$ be smooth and proper schemes over a field $F$ of characteristic zero.
 Then $X$ and $X'$ are stably birational if and only if their classes $[X]$ and $[X']$ in $\gro(\Var_F)$ are congruent modulo $\LL$.

In particular,
 $[X]$  is congruent to an integer $c$ modulo $\LL$ if and only if each of its connected components is stably rational; in that case, $c$ is the number of connected components of $X$.
\end{corollary}
\begin{proof}
Using the scissor relations in the Grothendieck ring, we can write the class of any $F$-scheme of finite type as the sum of the classes of its connected components. Now the statement follows immediately from Theorem \ref{thm:LL}.
\end{proof}


\section{Motivic volume and motivic reduction}\label{sec:vol}
In this section, we construct two specialization morphisms for Grothendieck rings of varieties: the motivic volume and the motivic reduction. In combination with the theorem of Larsen and Lunts, they will allow us to control the specialization of stable rationality in families.
\subsection{The motivic nearby fiber and the motivic volume}

The following result is heavily inspired by Denef and Loeser's motivic nearby fiber \cite{DL} and Hrushovski and Kazhdan's theory of motivic integration \cite{HK}; the precise relations will be explained below.

\begin{thm}\label{thm:vol}
There exists a unique ring morphism
$$\Vol_{K}:\gro(\Var_{K})\to \gro^{\gal}(\Var_k)$$ such that, for every smooth and proper $K$-scheme
$X$ and every snc-model $\cX$ of $X$, with
$\cX_k=\sum_{i\in I}N_i E_i,$
we have \begin{equation}\label{eq:volX}
\Vol_{K}(X)=\sum_{\emptyset \neq J\subset I}(1-\LL)^{|J|-1}[\widetilde{E}_J^o].
 \end{equation}
In particular, if $\cX$ is a smooth $R$-model of $X$, then
$\Vol_K(X)=[\cX_k]$ where $\gal$ acts trivially on $\cX_k$. The morphism $\Vol_K$ maps $\LL$ to $\LL$.

 Moreover, for every positive integer $n$ and every $K$-scheme $Y$ of finite type, $\Vol_{K(n)}(Y\times_K K(n))$ is the image of $\Vol_{K}(Y)$ under the ring morphism
$$\Res^{\gal}_{\gal(n)}:\gro^{\gal}(\Var_k)\to \gro^{\gal(n)}(\Var_k)$$
defined by restricting the $\gal$-action to the kernel $\gal(n)$ of the projection morphism $\gal\to \mu_n$.
\end{thm}
\begin{proof}
Uniqueness follows from the fact that the classes of smooth and proper $K$-schemes generate the group $\gro(\Var_K)$, by resolution of singularities. If $\Vol_K$ exists, then it maps $[\mathbb{P}^1_K]$ to $[\mathbb{P}^1_k]$ and $[\Spec K]$ to $[\Spec k]$, and thus $[\A^1_K]$ to $[\A^1_k]$ by additivity.

 One can deduce the existence of $\Vol_K$ from Bittner's presentation of the Grothendieck group (Theorem \ref{thm:bittner}).
 One uses the formula \eqref{eq:volX} as the definition of the motivic volume and applies Weak Factorization to prove that it is independent of the chosen snc-model and satisfies the desired properties. This argument will be explained in detail in Appendix  \ref{sec:WF}; see Theorem \ref{thm:vol2}.
\end{proof}
\begin{remark}
If $k$ contains all roots of unity, then Theorem \ref{thm:vol} follows from Hrushovski and Kazhdan's construction of the motivic volume $$\gro(\VF_{K})\to \gro^{\gal}(\Var_k).$$
 Here $\gro(\VF_{K})$ is the Grothendieck ring of semi-algebraic sets over $K$.
 Hrushovski and Kazhdan's construction in \cite{HK} is based on model theory of algebraically closed fields; a geometric translation can be found in \cite[2.5.1]{NiPa}. There is a natural morphism
 $$\gro(\Var_K)\to \gro(\VF_{K}):[X]\mapsto [X(K^a)]$$ obtained by viewing $X(K^a)$ as a semi-algebraic set defined over $K$, where $K^a$ denotes an algebraic closure of $K$.
 Composing these morphisms, we obtain a ring morphism
 $$\Vol_K:\gro(\Var_{K})\to \gro^{\gal}(\Var_k).$$ The formula \eqref{eq:volX} for $\Vol_{K}(X)$ is proven in \cite[2.6.1]{NiPa}, and the compatibility with the extensions $K(n)/K$ follows at once from the characterization of Hrushovski and Kazhdan's motivic volume in \cite[2.5.1]{NiPa} (see also the discussion at the end of \S2.5 in \cite{NiPa}). This argument can be generalized to the case where $k$ is an arbitrary field of characteristic zero, but this requires  some additional descent arguments to extend the comparison statement for Grothendieck rings in \cite[4.3.1]{HL} and \cite[2.7]{forey}, and the calculation on a strict normal crossings model in \cite[2.6.1]{NiPa}.
 \end{remark}

\begin{cor}\label{cor:vol}
 There exists a unique ring morphism
$$\Vol:\gro(\Var_{\pfi})\to \gro(\Var_k)$$ such that, for every positive integer $n$, every smooth and proper $K(n)$-scheme $X$ and every snc-model $\cX$ of $X$ over $R(n)$, with
$\cX_k=\sum_{i\in I}N_i E_i,$
we have $$\Vol(X\times_{K(n)}\pfi)=\sum_{\emptyset \neq J\subset I}(1-\LL)^{|J|-1}[\widetilde{E}_J^o].$$ This morphism maps $\LL$ to $\LL$.
\end{cor}
\begin{proof}
The Grothendieck ring $\gro(\Var_{\pfi})$ is the direct limit of the Grothendieck rings $\gro(\Var_{K(n)})$, by \cite[3.4]{NiSe-K0}.
Thus the result follows from Theorem \ref{thm:vol}.
\end{proof}

We will call the morphisms $\Vol_K$ and $\Vol$ from Theorem \ref{thm:vol} and Corollary \ref{cor:vol} the {\em motivic volume} maps.
 They are closely related to the motivic nearby fiber that was introduced by Denef and Loeser \cite{DL}. Let $Z$ be a smooth $k$-variety and let
  $$f:Z\to \A^1_k=\Spec k[t]$$ be a dominant morphism. We denote by $Z_0$ the zero locus of $f$.
The motivic nearby fiber $\psi^{\mathrm{mot}}_f$ is an element of the Grothendieck ring $\gro^{\gal}(\Var_{Z_0})[\LL^{-1}]$
of varieties over $Z_0$ with good $\gal$-action, localized with respect to $\LL$. It should be viewed as the motivic incarnation of the complex of
nearby cycles of $f$. In \cite[3.5.3]{DL}, Denef and Loeser have given an explicit formula
for $\psi^{\mathrm{mot}}_f$ in terms of a log resolution for the pair $(Z,Z_0)$. Comparing this formula to equation \eqref{eq:volX}, one sees that when $f$ is proper,
$$\psi_f^{\mathrm{mot}}=\Vol(X\times_{k[t]}K)$$ in the ring $\gro^{\gal}(\Var_k)[\LL^{-1}]$ (here we forget the $Z_0$-structure on the left hand side). Note, however, that for our purposes, it is crucial to avoid the inversion of $\LL$, since we want to control the reduction modulo $\LL$
in order to apply Larsen and Lunts' theorem (Theorem \ref{thm:LL}).

\subsection{The motivic reduction}

\begin{prop}\label{prop:red}
There exists a unique group morphism
$$\red:\gro(\Var_{K})\to \gro(\Var_k)$$ such that, for every smooth and proper $K$-scheme $X$ and every snc-model $\cX$ of $X$, with
$\cX_k=\sum_{i\in I}N_i E_i,$
we have \begin{equation}\label{eq:redX}
\red(X)=\sum_{\emptyset \neq J\subset I}(1-\LL)^{|J|-1}[E_J^o].
\end{equation}
 This is a morphism of $\Z[\LL]$-modules.
Moreover, for every regular model $\cY$ of $X$ over $R$,
$$\red(X)\equiv [\cY_k]\mod \LL.$$
\end{prop}
\begin{proof}
Uniqueness follows from the fact that the classes of smooth and proper $K$-schemes generate the group $\gro(\Var_K)$, by resolution of singularities.

To construct the morphism $\red$, we can simply compose the motivic volume
$\Vol_K$ with the quotient morphism of $\Z[\LL]$-modules
$$(\cdot)/\gal:\gro^{\gal}(\Var_k)\to \gro(\Var_k):[X]\mapsto [X/\gal]$$
whose existence was proven in \cite[5.1]{looijenga} (there it was assumed that $k$ is algebraically closed, but the proof goes through for an arbitrary field $k$ of characteristic zero).

 If $\cY$ is a regular model for $X$ over $R$, then by Hironaka's resolution of singularities, we can turn $\cY$ into an snc-model by means of a finite sequence of blow-ups
 with smooth centers contained in the special fiber. Such a blow-up does not change the class of the special fiber modulo $\LL$, since the exceptional divisors are the projectivizations of the normal bundles of the centers. Thus we may assume that $\cY$ is an snc-model for $X$. In that case, applying equation \eqref{eq:redX} to $\cY$ and reducing both sides modulo $\LL$,
 we find that $\red(X)\equiv [\cY_k]$ modulo $\LL$.
\end{proof}

The existence of the morphism $\red$ can also be deduced from Bittner's presentation and Weak Factorization; see Proposition \ref{prop:motred2}.
 We call the morphism $\red$ the {\em motivic reduction} map.
 A weaker version of this map (only working modulo $\LL$ and on objects defined over the function field $k(t)$) was constructed in \cite[9.4]{hartmann}.
  The motivic volume and the motivic reduction fit into the following commutative diagram:
$$\xymatrix{
\gro(\Var_{\pfi}) \ar[r]^{\Vol}  &\gro(\Var_k)\\
\gro(\Var_{K}) \ar[u] \ar[rd]_{\red} \ar[r]^{\Vol_K} & \gro^{\gal}(\Var_k) \ar[u]_{\Res^{\gal}_{\{1\}}} \ar[d]^{(\cdot)/\gal}
\\ & \gro(\Var_k)}$$
Here the left vertical map is the base change morphism, $(\cdot)/\gal$ is the quotient morphism and $\Res^{\gal}_{\{1\}}$ is the restriction morphism that forgets the $\gal$-action.

 Note that, when $X$ is a smooth and proper $K$-scheme with semistable reduction, we have
 $$\Vol_K(X)=\Vol(X\times_K \pfi)=\red(X)$$
 in $\gro^{\gal}(\Var_k)$, where we view the first and third members of the equality as elements of $\gro^{\gal}(\Var_k)$ {\em via} the map
 $$\Res^{\{1\}}_{\gal}:\gro(\Var_k)\to \gro^{\gal}(\Var_k)$$
 that endows $k$-varieties with the trivial $\gal$-action. Thus it follows from the Semistable Reduction Theorem \cite{Mumford} that for every $K$-scheme of finite type $Y$, there exists a positive integer $n$ such that
 $$\Vol_{K(n)}(Y\times_K K(n))=\Vol(Y\times_K \pfi)=\red(Y\times_K K(n)).$$

\begin{remark}\label{rem:MRhom}
Beware that the morphism $\red$ is {\em not} multiplicative. For instance, if $X=\Spec k\llpar \sqrt{t} \rrpar$ then
is is easy to check that $\red(X)=1$ whereas $\red(X\times_K X)=2$. The issue is that the quotient morphism $(\cdot)/\gal$ is not multiplicative.
However, $\red$ does preserve some additional structure.
Let $\gro(\Var_K)^{\mathrm{ss}}$  be the subring of $\gro(\Var_K)$
 consisting of the elements whose image under $\Vol_K$ lies in the image of the injective ring morphism  $$\Res^{\{1\}}_{\gal}:\gro(\Var_k)\to \gro^{\gal}(\Var_k).$$ For instance, if $Y$ is a smooth and proper $K$-scheme with semistable reduction, then $[Y]$ lies in $\gro(\Var_K)^{\mathrm{ss}}$.
 We have $\Vol_K=\red$ on $\gro(\Var_K)^{\mathrm{ss}}$, and if we view $\gro(\Var_k)$ as a $\gro(\Var_K)^{\mathrm{ss}}$-algebra {\em via} the morphism $\Vol_K$, then $\red$ is a morphism of $\gro(\Var_K)^{\mathrm{ss}}$-modules.
\end{remark}

\section{Applications to stable rationality}\label{sec:rat}

\subsection{Specialization of stable birational equivalence}
The starting point of our applications to rationality questions is the following statement, which can be viewed as an obstruction to stable rationality of smooth and proper $\pfi$-schemes.
\begin{theorem}\label{thm:volrat}
Let $X$ and $Y$ be smooth and proper $\pfi$-schemes. If $X$ is stably birational to $Y$, then
$$\Vol(X)\equiv \Vol(Y) \mod \LL.$$ In particular, if $X$ is stably rational, then
$$\Vol(X)\equiv 1\mod \LL.$$
\end{theorem}
\begin{proof}
This follows immediately from Corollary \ref{cor:LL} and the fact that $\Vol$ is a ring morphism that sends $[\A^1_{\pfi}]$ to $[\A^1_k]$.
\end{proof}
As an immediate consequence, we obtain the following specialization property for stable birational equivalence.
\begin{theorem}[Specialization of stable birational equivalence]\label{thm:specbir}
Let $\cX$ and $\cY$ be smooth and proper $R$-schemes. If $\cX_{\pfi}$ is stably birational to $\cY_{\pfi}$, then $\cX_k$ is stably birational to $\cY_k$.
 In particular, if $\cX_{\pfi}$ is stably rational, then $\cX_k$ is stably rational, as well.
\end{theorem}
\begin{proof}
This follows from Theorem \ref{thm:volrat} and Theorem \ref{thm:LL}, because $\Vol(\cX_{\pfi})=[\cX_k]$ and $\Vol(\cY_{\pfi})=[\cY_k]$ by the definition of the motivic volume.
\end{proof}

This result can be generalized to strict normal crossings degenerations in the following way. An application of this generalization will be given in Section \ref{subsec:app}.

\begin{theorem}\label{thm:dual}
Let $X$ be a smooth and proper $K$-scheme and let $\cX$ be an snc-model for $X$, with $\cX_k=\sum_{i\in I}N_i E_i$.
 Assume that every connected component
of $\widetilde{E}_J^o$ is stably rational, for every subset $J$ of $I$ of cardinality at least $2$.
 If $X_{\pfi}$ is stably rational, then all the connected components of all the covers $\widetilde{E}^o_{i}$ are stably rational.
\end{theorem}
\begin{proof}
  Let $m$ be a positive integer and denote by $\cX(m)$ the normalization of
 $\cX\times_{R}R(m)$.
The Semistable Reduction Theorem tells us that, if $m$ is sufficiently divisible, then we can find a toroidal modification $\cY\to \cX(m)$
such that $\cY$ is a semi-stable model for $X\times_K K(m)$.
 Moreover, for every non-empty subset $J$ of $I$, the cover $\widetilde{E}^o_J$ is the reduced inverse image of $E_J^o$ in $\cX(m)_k$, by \cite[3.2.2]{BuNi}.

We write $\cY_k=\sum_{i\in I'}E'_i$.
 Then, for every subset $J'$ of $I'$ of cardinality at least $2$, each connected component of  $E'_{J'}$ is stably rational, because it is birational to a connected component of
 $\widetilde{E}_J^o\times_k \P^\ell_k$ for some subset $J$ of $I$ of cardinality at least $2$ and some $\ell\geq 0$.
  Moreover, for every $i$ in $I$, the cover $\widetilde{E}^o_{i}$ is isomorphic to a disjoint union of strata $(E'_{i'})^o$ with
 $i'\in I'$.
    Thus, replacing $\cX$ by $\cY$, we may assume that $\cX_k$ is reduced; then $\wt{E}^o_J = E^o_J$ for all non-empty subsets $J$ of $I$.
 We denote by $c_J$ the number of connected components of $E_J$.

 The motivic volume of $X_{\pfi}$ is congruent modulo $\LL$  to
$$\sum_{\emptyset \neq J\subset I}[E_J^o] = \sum_{\emptyset \neq J\subset I}(-1)^{|J|-1}[E_J]$$ where the equality follows
 from a straightforward inclusion-exclusion argument.
 Since $[E_J]\equiv c_J$ modulo $\LL$ as soon as $|J|\geq 2$, we can rewrite this expression as
 $$\chi(\Delta)-|I|+\sum_{i\in I}[E_i]$$
 where $\Delta$ is the dual intersection complex of $\cX_k$ and $\chi(\Delta)$ denotes its Euler characteristic.
 On the other hand, if $X_{\pfi}$ is stably rational, then its motivic volume is congruent to $1$ modulo $\L$,
 so that $\sum_{i\in I}[E_i]=[\sqcup_{i\in I} E_i]$ is congruent to an integer
 modulo $\L$. Corollary \ref{cor:LL} now implies that $E_i$ is stably rational for all $i$ in $I$.
\end{proof}

Theorem \ref{thm:specbir} has the following interesting geometric consequence.
\begin{theorem}\label{thm:countable}
Let $S$ be a Noetherian $\Q$-scheme, and let $f:X\to S$ and $g:Y\to S$ be smooth and proper
 morphisms. Let $S_{\mathrm{sb}}$ be the set of points $s$ in $S$ such that the fibers of $f$ and $g$ over $\overline{s}$ are stably birational, for any geometric point $\overline{s}$ based at $s$. Then
 $S_{\mathrm{sb}}$ is a countable union of closed subsets of $S$.
 \end{theorem}
\begin{proof}
By the same proof as for Proposition 2.3 in \cite{dFF}, the set $S_{\mathrm{sr}}$ is a countable union of {\em locally} closed subsets in $S$
 (in \cite{dFF} the authors
  only consider closed points and assume that $S$ is of finite type over an algebraically closed field, but the proof yields this more general result; we can reduce to the case where $f$ and $g$ are projective by replacing $X$ and $Y$ by birationally equivalent smooth projective families, up to a finite partition of $S$ into subschemes).

 Thus it suffices to show that the set $S_{\mathrm{sb}}$ is closed under specialization.
 For this we can reduce to the case where $S$ is an integral local $\Q$-scheme of dimension one.
 The stable birationality of geometric fibers is invariant under arbitrary base change.
 By base change to the normalization of $S$ and localizing at a closed point, we may assume that $S$ is regular (here we use that the normalization of a Noetherian domain of dimension one is again Noetherian, by the Krull-Akizuki Theorem). Then $S$ is the spectrum of a discrete valuation ring $A$ of equal characteristic zero, and the completion of $A$ is isomorphic to a power series ring $F\llbr t\rrbr$ where $F$ is a field of characteristic zero. Now the result follows from  Theorem \ref{thm:specbir}.
\end{proof}

\begin{cor}\label{cor:stabratlocus}
Let $S$ be a Noetherian $\Q$-scheme, and let $f:X\to S$ be a smooth and proper
 morphism. Let $S_{\mathrm{sr}}$ be the set of points $s$ in $S$ such that the fiber of $f$ over $\overline{s}$ is stably rational, for any geometric point $\overline{s}$ based at $s$.
  Then $S_{\mathrm{sr}}$ is a countable union of closed subsets of $S$.

  In particular, if $S$ is integral and the geometric generic fiber of $f$ is stably rational, then every geometric fiber of $f$ is stably rational.
\end{cor}
\begin{proof}
This follows from Theorem \ref{thm:countable} by taking $Y=S$.
\end{proof}

\begin{cor}\label{cor:verygen}
Assume that $k$ is uncountable and algebraically closed. Let $S$ be a $F$-scheme of finite type, and let $f:X\to S$ be a smooth and proper morphism. If a very general closed fiber of $f$ is stably rational, then every  closed fiber of $f$ is stably rational.
\end{cor}
\begin{proof}
This follows immediately from
Corollary \ref{cor:stabratlocus}.
\end{proof}

\subsection{$\LL$-rational singularities and $\LL$-faithful models}
 The aim of this section is to generalize Corollary \ref{cor:verygen} to families with mildly singular fibers. We will
 consider a class of singularities characterized by the following definition.

 \begin{definition}\label{def:ordinary}
Let $Y$ be a reduced $k$-scheme of finite type and let $y$ be a singular point of $Y$. We say that $Y$ has an {\em ordinary double point} at $y$
 if, Zariski-locally around $y$, the singular locus $Y_{\mathrm{sing}}$ of $Y$ is smooth, and the projectivized normal cone of $Y_{\mathrm{sing}}$ in $Y$ is a smooth quadric bundle over
 $Y_{\mathrm{sing}}$ that has a section.
 \end{definition}

   Note that we do not require $y$ to be a closed point of $Y$, and that ordinary double points are not necessarily isolated singularities. The definition also includes the case where $Y_{\mathrm{sing}}$ has codimension one in $Y$; then the projectivized normal cone is simply a trivial degree two cover of $Y_{\mathrm{sing}}$. If $Y$ has only ordinary double points as singularities, then the blow-up of $Y$ along $Y_{\mathrm{sing}}$ is a resolution of singularities for $Y$, and we can identify the exceptional divisor of this resolution with the projectivized normal cone of $Y_{\mathrm{sing}}$ in $Y$. The following result gives a characterization of ordinary double points that are hypersurface singularities.

\begin{prop}\label{prop:hypersurf}
Let $Y$ be a reduced $k$-scheme of finite type.

\begin{enumerate}
\item \label{it:odp}
 Assume that at every singular point $y$ of $Y$,  the completed local ring is of the form
 $$\widehat{\mathcal{O}_{Y,y}}\cong \kappa(y)\llbr z_0,\ldots,z_n\rrbr/(q(z_0,\ldots,z_n))$$ for some integer $n> 0$ and some isotropic quadratic form $q$ over $\kappa(y)$ of rank at least $2$. Then all the singular points of $Y$ are ordinary double points.

\item \label{it:q}  Let $y$ be an ordinary double point of $Y$, let $n$ be the dimension of $\mathcal{O}_{Y,y}$, and assume that the embedding dimension of $\mathcal{O}_{Y,y}$ equals $n+1$. Denote by $d$ the codimension of the singular locus of $Y$ at the point $y$.
Then there exists a surjective morphism of $k$-algebras
$$\varphi:\kappa(y)\llbr z_0,\ldots,z_n \rrbr \to  \widehat{\mathcal{O}_{Y,y}}.$$
 For every such morphism $\varphi$, there exists a $\kappa(y)$-automorphism $\theta$ of $\kappa(y)\llbr z_0,\ldots,z_n \rrbr$ such that
 the kernel of $\varphi\circ \theta$ is generated by an isotropic quadratic form $q$ over $\kappa(y)$ of rank $d+1$.
\end{enumerate}
\end{prop}
\begin{proof}
 \eqref{it:odp} The singular locus of $Y$ is smooth and the projectivized normal cone of $Y_{\mathrm{sing}}$ in $Y$ is a smooth quadric bundle over
 $Y_{\mathrm{sing}}$, because these properties can be checked on the completed local rings.  {\em A priori}, the description of the completed local rings only provides {\em formal} local sections for the projectivized normal cone of the singular locus; but for quadric bundles, already the existence of a rational section implies the existence of Zariski-local sections, by \cite{panin}. Thus by using the isotropy of $q$ at the generic points of the singular locus of $Y$, we find that the projectivized normal cone of the singular locus has sections Zariski-locally.

 \eqref{it:q}  Since the embedding dimension at $y$ is equal to $n+1$, we can
  find elements $(z_0,\ldots,z_n)$ in the maximal ideal $\mathfrak{m}_y$ of $\mathcal{O}_{Y,y}$ whose residue classes generate the vector space
  $\frak{m}_y/\frak{m}_y^2$ over $\kappa(y)$. Then the choice of a section of the residue morphism of $k$-algebras $\widehat{\mathcal{O}}_{Y,y}\to \kappa(y)$ determines a surjective morphism of $k$-algebras
$$\varphi:\kappa(y)\llbr z_0,\ldots,z_n \rrbr \to  \widehat{\mathcal{O}_{Y,y}}.$$
 Its kernel is generated by a power series $f(z_0,\ldots,z_n)$.
    We need to find a change of coordinates that transforms $f$ into a quadratic form.

    If $y$ is an isolated singularity then the existence of such a coordinate transformation follows from the Morse Lemma: the assumption that the projectivized tangent cone is a smooth quadric is equivalent with the property that $f$ has a non-degenerate critical point at the origin.
  The general case can be proven in the following way. By our hypothesis that the singular locus of $Y$ is smooth, we may assume that it is defined by
  the equations $$z_0=\ldots=z_{d}=0$$ with $1\leq d\leq n$.
  If we set $$g(z_0,\ldots,z_d)=f(z_0,\ldots,z_d,0,\ldots,0),$$
  then the normal cone to the singular locus $Y_{\mathrm{sing}}$ at the point $y$ is given by $\mathrm{Proj}\,\kappa(y)[z_0,\ldots,z_d]/(q)$ where
  $q$ consists of the lowest degree terms in $g$. Because we are assuming that this normal cone is a smooth quadric, the polynomial $q$ is a non-degenerate quadratic form in $z_0,\ldots,z_d$. Applying the Morse Lemma to $g$, we can arrange by means of a coordinate transformation on $(z_0,\ldots,z_d)$ that $g$ takes the form $a_0z_0^2+\cdots+a_dz_d^2$ with $a_0,\ldots,a_d$ non-zero elements in $\kappa(y)$.

Looking at the equations for the critical locus of $f$, we see that each term of the power series $f-g$ is at least quadratic in the variables $(z_0,\ldots,z_d)$.
 We can get rid of mixed quadratic terms of the form $cz_iz_jh(z_{d+1},\ldots,z_n)$, with $c$ in $\kappa(y)$ and $i,j$ distinct elements in $\{0,\ldots,d\}$, by completing squares:
 we write
 $a_iz_i^2 +  cz_iz_jh(z_{d+1},\ldots,z_n)$ as $$a_i(z_i+\frac{c}{2a_i}z_jh(z_{d+1},\ldots,z_n))^2 - (\frac{c}{2a_i}z_jh(z_{d+1},\ldots,z_n))^2$$
 and we perform a coordinate change $$z'_i= z_i+\frac{c}{2a_i}z_jh(z_{d+1},\ldots,z_n).$$
 Note that this does not affect the equations for the critical locus of $f$. Repeating this operation, we arrive at an expression where  $f-g$ is a sum of terms of the form
 $$c'z_i^2h'(z_0,\ldots,z_n)$$ with $c'$ in $\kappa(y)$ and where $h'$ has no constant term. Then we write
 $$a_iz_i^2+ c'z_i^2h'(z_0,\ldots,z_n)=a_iz^2_i(1+\frac{c'}{a_i}h')$$ and we perform a coordinate change
 $$z_i'=z_i(1+\frac{c'}{a_i}h')^{1/2}.$$
These operations reduce $f$ to a diagonal quadratic form of rank $d+1$ over $\kappa(y)$; it is isotropic because the normal cone to $Y_{\mathrm{sing}}$ at $y$ is assumed to have a section.
\end{proof}

\begin{remark}
 Proposition \ref{prop:hypersurf}\eqref{it:odp} is false for individual points $y$: knowing the completed local ring of $Y$ at $y$ is not sufficient to conclude the
existence of a section of the projectivized normal cone along the singular locus Zariski-locally around $y$ (unless $y$ is an isolated singularity).
\end{remark}


The following definition is an analog of rational singularities in the context of the Grothendieck ring of varieties.

\begin{definition}\label{def:Lrat}
Let $Y$ be an integral $k$-scheme of finite type. Let $y$ be a point of $Y$ and let $\kappa(y)$ denote its residue field. We say that $Y$ has an {\em $\LL$-rational singularity} at $y$ if there exists a resolution of singularities $h:Y'\to (Y,y)$ of the germ $(Y,y)$ such that $[h^{-1}(y)] \equiv 1 \mod \L$ in $\gro(\Var_{\kappa(y)})$. We say that $Y$ has {\em $\LL$-rational singularities} if $Y$ has an $\LL$-rational singularity at every point $y$ of $Y$.
\end{definition}

If $Y$ has an $\LL$-rational singularity at $y$, then it follows easily from the Weak Factorization Theorem \cite{WF, Wlod} that
 $[h^{-1}(y)] \equiv 1 \mod \L$ in $\gro(\Var_{\kappa(y)})$ for {\em every} resolution of singularities $h:Y'\to (Y,y)$ of the germ $(Y,y)$.
 If $Y$ has $\LL$-rational singularities, then the following lemma implies that $[Y']\equiv [Y]\mod \LL$ in $\gro(\Var_k)$ for every resolution of singularities $Y'\to Y$.
This property was taken as the definition of $\LL$-rational singularities in \cite{Huh}, but it has the drawback of not being of local nature; for instance,
 according to that definition, $X\times_k \A^1_k$ has $\LL$-rational singularities for every $k$-variety $X$. This is why we have opted to work with Definition \ref{def:Lrat} instead, which is local on $Y$
 and more restrictive than \cite[Definition 6]{Huh}.

\begin{lemma}\label{lemma:spreadK0}
Let $X$ and $Y$ be $k$-schemes of finite type and let $h\colon Y\to X$ be a morphism of $k$-schemes.
 Assume that $[h^{-1}(x)] \equiv 1 \mod \L$ in $\gro(\Var_{\kappa(x)})$ for every point $x$ of $X$, where $\kappa(x)$ denotes the residue field of $X$ at $x$.
  Then $[Y]\equiv [X]\mod \L$ in $\gro(\Var_k)$.
\end{lemma}
\begin{proof}
For every point $x$ of $X$, there exists a subscheme $U$ of $X$ containing $x$ such that $[h^{-1}(U)] \equiv [U] \mod \L$ in the Grothendieck ring of $U$-varieties $\gro(\Var_U)$, by \cite[3.4]{NiSe-K0}. The result now follows from additivity and noetherian induction.
\end{proof}

\begin{examples}\label{exam:Lrat}
\begin{enumerate}
\item Assume that $k$ is algebraically closed. Let $Y$ be a normal surface over $k$ and let $y$ be a point of $Y$. Then $Y$ has an $\LL$-rational singularity at $y$ if and only if $Y$ has a rational singularity at $y$.
 To see this, it suffices to observe that when $D$ is a connected strict normal crossings divisor on a smooth $k$-surface, then $[D]$ is congruent to $1$ modulo $\LL$ in $\gro(\Var_k)$
 if and only if $D$ is a tree of rational curves.

 \item \label{it:Lratdouble}   Let $Y$ be an integral $k$-scheme of finite type. Let $y$ be an ordinary double point of $Y$ such that, locally around $y$, the singular locus of $Y$ has codimension at least $2$. Then $Y$ has an $\LL$-rational singularity at $y$. Indeed, blowing up
        $Y$ at its singular locus resolves the singularity at $y$, and the fiber over $y$ is a smooth projective quadric $Q$ over $\kappa(y)$ of positive  dimension with a rational point.
        For such a quadric $Q$, we have $[Q]\equiv 1$ modulo $\LL$ in $\gro(\Var_{\kappa(y)})$

 \item Let $Y$ be an integral $k$-scheme of finite type. We say that $Y$ has {\em strictly toroidal} singularities if
 we can cover $Y$ by open subschemes that admit an \'etale morphism to a toric $k$-variety. This definition is stronger than the usual definition of toroidal singularities because it is not local with respect to the \'etale topology.
  If $Y$ has strictly toroidal singularities, then it has $\LL$-rational singularities: one can immediately reduce to the toric case, which is
 straightforward.
  \end{enumerate}
\end{examples}

\begin{definition}
Let $X$ be a smooth and proper $K$-scheme, and let $\cX$ be an $R$-model of $X$. We say that $\cX$ is {\em $\LL$-faithful} if
$$[\cX_k]\equiv \Vol(X_{\pfi})\mod \LL.$$
\end{definition}

\begin{exam}
If $\cX$ is an snc-model of $X$ and $\cX_k$ is reduced, then $\cX$ is $\LL$-faithful. This follows immediately from the explicit expression
for $\Vol(X_{\pfi})$ in terms of $\cX$ in Corollary \ref{cor:vol}.
\end{exam}

\begin{prop}\label{prop:odp}
Let $X$ be a smooth proper $K$-scheme, and let $\cX$ be a regular $R$-model of $X$.
Assume that $\cX_k$ is reduced and has at most ordinary double points as singularities. Then $\cX$ is $\LL$-faithful.
\end{prop}
\begin{proof}
To compute the motivic volume $\Vol(X)$ we make a ramified degree two base change: we set $R'=R(2)$ and $K'=K(2)$  and we define $X' = X\times_K K'$, $\cX' = \cX \times_R R'$. Then $\cX'$ is normal but not regular, in general: it has an ordinary double point over each singular point of $\cX_k$
(this follows from the explicit equation in Proposition \ref{prop:hypersurf}\eqref{it:q}).
 Blowing up $\cX'$ at its singular locus, we obtain an snc-model $\cY$ of $X'$ with reduced special fiber $\cY_k$.
 The fiber of $\cY\to \cX'$ over each singular point $x$ of $\cX'$ is a smooth projective quadric $Q$ of dimension $\geq 1$ over the residue field $\kappa(x)$ with a rational point.
 For such a quadric $Q$, we have $[Q]\equiv 1$ modulo $\LL$ in $\gro(\Var_{\kappa(x)})$, and it follows that
 $$[\cY_k]\equiv [\cX'_k]\equiv [\cX_k]\mod \LL$$ in $\gro(\Var_k)$ by Lemma \ref{lemma:spreadK0}. Since $\Vol(X)\equiv [\cY_k]$ modulo $\LL$, we find that $\cX$ is $\LL$-faithful.
\end{proof}

The importance of $\LL$-rational singularities and $\LL$-faithful models lies in the following property.

\begin{prop}\label{prop:st-rat}
Let $X$ be a smooth and proper $K$-scheme, and let $\cX$ be an $\LL$-faithful $R$-model of $X$ such that $\cX_k$ is integral and has $\LL$-rational singularities.
 If
$\cX_{\pfi}$ is stably rational, then $\cX_k$ is stably rational.
\end{prop}
\begin{proof} If $\cX_{\pfi}$ is stably rational, then
$\Vol(\cX_{\pfi})\equiv 1$ modulo $\LL$ by Theorem \ref{thm:volrat}.
 Thus $[\cX_k]\equiv 1$ modulo $\LL$ by the definition of an $\LL$-faithful model, and
 $[Y]\equiv 1$ modulo $\LL$ for any resolution of singularities $Y\to \cX_k$
 by the definition of $\LL$-rational singularities. From the theorem of Larsen and Lunts (Theorem \ref{thm:LL}), we now deduce that $Y$, and thus $\cX_k$, are stably rational.
\end{proof}

\begin{theorem}\label{thm:nonrat}
Let $f\colon \cX \to C$ be a proper flat morphism with $\cX$, $C$ connected smooth $k$-schemes and $\dim(C) = 1$.
Assume that the geometric generic fiber of $f$ is stably rational. Then all geometric fibers  whose only singularities are ordinary double points have a stably rational irreducible component.
\end{theorem}
 If $k$ is uncountable and algebraically closed, then we get an equivalent statement by replacing
``the geometric generic fiber'' by  ``a very general closed fiber,'' by Corollary \ref{cor:stabratlocus}.

\begin{proof}
By base change to an algebraic closure of $k$, we can reduce to the case where $k$ is algebraically closed. Let $s \in C$ be a closed point such that $f^{-1}(s)$ has at most ordinary double points as singularities.
 We identify $ \OO_{C,s}$ with $R$ by choosing a uniformizer $t$.
Then $\cX_R=\cX\times_{C}\Spec R$ is a regular model for $\cX_K=\cX\times_{C}\Spec K$
 whose special fiber is canonically isomorphic to $f^{-1}(s)$.
  The scheme $\cX_{\pfi}$ is obtained from the geometric generic fiber of $f$ by extension of scalars; thus it is stably rational.
 By Proposition \ref{prop:odp}, we know that $\cX$ is $\LL$-faithful, so that
 $$1\equiv \Vol(\cX_{K(\infty)})\equiv [\cX_k]\mod \LL.$$
  We cannot directly apply Proposition \ref{prop:st-rat} because $\cX_k$ does not have $\LL$-rational singularities if it is reducible. However, if we denote by $Y_1,\ldots,Y_r$ the irreducible components of $\cX_k$, then the following properties follow easily from the definition of an ordinary double point:
  \begin{itemize}
\item  each component $Y_i$ has at most ordinary double points as singularities;
\item any nonempty intersection of two distinct components $Y_i$ and $Y_j$ is smooth and of pure codimension one;
\item there are no triple intersection points, that is, the intersection of any three distinct components $Y_i$, $Y_j$ and $Y_\ell$ is empty.
\end{itemize}
  By the scissor relations in the Grothendieck ring, we have
  $$[\cX_k]=\sum_{i} [Y_i]-\sum_{i<j}[Y_i\cap Y_j],$$
  so that
  $$1\equiv \sum_{i} [Y_i]-\sum_{i<j}[Y_i\cap Y_j]\mod \LL.$$
For each component $Y_i$,  we denote by  $Y'_i\to Y_i$ the birational
  modification obtained by blowing up all the connected components of the singular locus of codimension at least $2$. All these singularities are resolved by the blow-up $Y'_i\to Y_i$, and they are $\LL$-rational by Example \ref{exam:Lrat}\eqref{it:Lratdouble}, so that
  $[Y'_i]\equiv [Y_i]$ modulo $\LL$ by Lemma \ref{lemma:spreadK0}. Denote by $Z_i$ the singular locus of $Y'_i$ (which now is empty or of pure codimension one), and by $Y''_i\to Y'_i$ the blow-up of $Y'_i$ along $Z_i$.
The schemes $Y''_i$ and $Z_i$ are smooth and proper over $k$, and $Y''_i$ is connected.  Lemma \ref{lemma:spreadK0} implies that $[Y''_i]=[Y'_i]+[Z_i]$.
 Thus, we find that
  \begin{equation}\label{eq:tree}
  1\equiv  \sum_{i} [Y''_i]-\sum_{i} [Z_i]-\sum_{i<j}[Y_i\cap Y_j]\mod \LL.
  \end{equation}
  By Theorem \ref{thm:LL}, this implies that at least one of the varieties $Y''_i$ is stably rational.
\end{proof}

\begin{remark}
A closer examination of the proof of Theorem \ref{thm:nonrat} reveals that all the schemes $Z_i$ are empty
 and the dual intersection graph $\Gamma$ of $\cY_k$ is a tree. Indeed, Theorem \ref{thm:LL} and Equation \ref{eq:tree} 
 imply that 
$$ 1 = \sum_{i} [Y''_i]-\sum_{i} [Z_i]-\sum_{i<j}\sum_{C\in \pi_0(Y_i\cap Y_j)}[C]$$
 in $\Z[\mathrm{SB}_k]$. Adding up the coefficients of nonzero terms on both sides, we get 
 $$1=\chi(\Gamma)-N$$ where $N$ is the number of nonempty schemes $Z_i$.
 Since $\Gamma$ is a connected graph, it follows that $N=0$ and $\Gamma$ is a tree.

Note that, in the statement of Theorem \ref{thm:nonrat}, the existence of one stably rational component is the best one can hope for: one can create stably irrational components by blowing up curves of positive genus in the fiber.
\end{remark}

\subsection{Applications}
\label{subsec:app}
We will discuss a few concrete applications to illustrate how our results can be used in practice.

\begin{theorem}\label{thm:hypers}
Assume that $k$ is algebraically closed. If there exists a single integral hypersurface $X_0 \subset \P^{n+1}_k$ of degree $d$
(resp. a degree $d$ cyclic covering $X_0 \to \P^n_k$)
with only isolated ordinary double points as singularities, and $X_0$ is not stably rational,
then a very general smooth hypersurface in $\P^{n+1}_k$ (resp.~cyclic covering of $\P^n_k$) of degree $d$  is not stably rational.
\end{theorem}
\begin{proof}
By Corollary \ref{cor:stabratlocus}, it suffices to find one smooth hypersurface in $\P^{n+1}_k$ (resp.~cyclic covering of $\P^n_k$) of degree $d$ that is not stably rational.
 We will apply Theorem \ref{thm:nonrat} to families $\cX \to C$ of varieties as in the statement over a
connected smooth $k$-curve $C$ such that one of the fibers is isomorphic to $X_0$ and the total space $\cX$ is regular.

We first consider the case of hypersurfaces. Let $X_1$ be a smooth hypersurface of degree $d$.
Let $F_0,F_1 \in H^0(\P^{n+1}_k, \cO(d))$ be equations for $X_0$ and $X_1$.
Consider the pencil $\cX\subset \P^{n+1}_k \times_k \A^1_k$ defined by the equation
\[
F_0 + t F_1 = 0.
\]
 Using the Jacobian criterion, one sees that $\cX$ is regular if $X_0$ and $X_1$ are not
tangent along their intersection. Since $X_0$ has isolated singularities, this will hold for general $X_1$ by Bertini's Theorem.

The case of cyclic coverings is quite similar. Let $D_0 \subset \P^n_k$ be the ramification divisor of $X_0 \to \P^n_k$.
 If $X_0$ has only isolated ordinary double points as singularities, then $D_0$ satisfies the same property.
 We embed $D_0$ in a pencil of hypersurfaces $\cD \subset \P^n_k \times_k \A^1_k$ with regular total space, as above.
Then the degree $d$ covering
ramified in $\cD \subset \P^n_k \times_k \A^1_k$ is a regular scheme such that one of the closed fibers over $\A^1_k$ is isomorphic to $X_0$.
\end{proof}

\begin{example}\label{ex:qdsolid}
The Artin-Mumford quartic double solid $X_0$, one of the first
examples of non-rational unirational varieties \cite{AM}, has $10$ isolated ordinary double points, and is not stably rational because it has non-vanishing unramified Brauer group.
It follows that very general smooth quartic double solids are not stably rational either. This has already been proved by Voisin \cite{Voisin} using the degeneration method for Chow groups of zero-cycles
and the same degeneration $X_0$. A similar argument applies to a smooth sextic double solid \cite{Beauville-stab}.
\end{example}

We now illustrate how our results can be applied in the case when the central fiber $X_0$ has non-isolated singularities which are more complicated than ordinary double points.
  We will show that very general smooth quartic threefolds are not stably rational; this result was first proven by Colliot-Th\'el\`ene and Pirutka in \cite{CTP-stab} by means of a refinement of Voisin's specialization method.

 Assume that $k$ is algebraically closed. We will construct an appropriate strict normal crossings
degeneration where one of the components of the central fiber is birational to Huh's quartic  $X_0 \subset \P^4_k$ \cite[Definition 4]{Huh}, and apply Theorem \ref{thm:dual}.
 By construction, $X_0$ has $9$ isolated ordinary double points and is also singular along a line $L \subset X_0$.
The singularities of $X_0$ along $L$ are quadratic; however, the rank of the quadric normal cone drops along $L$, so that these singularities do not satisfy Definition \ref{def:ordinary}.
 It is shown in \cite{Huh} that $X_0$ is not stably rational, has $\LL$-rational singularities\footnote{Huh's argument also applies to our more restrictive definition of $\LL$-rational singularities because, in his notation, the conic bundle $S\to L$ has a section.} and, hence, satisfies $[X_0]\not\equiv 1$ modulo $\LL$.

Similarly to the proof of Theorem \ref{thm:hypers}, we take $\cX$ to be the subscheme of  $\P^4_R$ defined by $F_0 + tF_1 = 0$, where $F_0=0$ is an equation for $X_0$ and
$F_1=0$ defines a general smooth quartic $X_1 \subset \P^4_k$.

\begin{prop}\label{prop:quart-model}
The variety $\cX_{\pfi}$ is not stably rational.
\end{prop}
\begin{proof}
The total space $\cX$ is nonsingular outside four isolated ordinary double points $P_1$, $P_2$, $P_3$, $P_4$ obtained as the intersection $X_1 \cap L$. Since $F_1$ is general, the points $P_i$ are general points on $L$, so that we may assume that they are ordinary double points of $X_0$. Thus we can use Proposition \ref{prop:hypersurf} to describe the formal structure of $X_0$ at the points $P_i$ and compute the blow-up of $X_0$ at these points.

 Let $\cY\to \cX$ denote the blow-up of the points $P_i$, followed by the blow-up of the proper transform $L'$ of $L$ and the blow-up of the nine isolated
ordinary double points of $X_0$.
Then $\cY$ is a regular model whose central fiber is the union of the components $X'_0$, $Q_1, \dots, Q_4$, $V_1, \dots, V_9$ and $E$.
 Here $X'_0$ is obtained from $X_0$ by successively blowing up the points $P_i$, the curve $L'$, and the nine isolated double points.
The $Q_i$ and $V_j$ are smooth rational threefolds that lie above the points $P_i$ and the $9$ isolated double points of $X_0$, respectively.
 The component $E$ is the exceptional divisor of the blow-up of $L'$; thus $E$ is a smooth $\P^2_k$-bundle over $L'\cong \P^1_k$.

 A local computation shows that $\cY$ is an snc-model for $\cX_K$ and the components of $\cY_k$ have multiplicity one, except for the $V_j$ and $E$ which have multiplicity two. 
    The surfaces $E\cap Q_i$ are fibers of the bundle $E\to L'$ for all $i$, and thus isomorphic to $\P^2_k$.
 The exceptional surfaces $Q_i\cap X_0'$ and $V_j\cap X_0'$ in $X_0'$ are smooth quadric surfaces, for all $i$ and $j$, and $E\cap X_0'$ is a quadric bundle over $L'$ (see also Lemma 9 in \cite{Huh} for a description away from the points $P_i$). The only non-empty triple intersections in the special fiber of $\cY$ are the curves
 $V_j\cap E\cap X_0'$ for $j\in \{1,\ldots,4\}$. These curves are smooth fibers of the quadric bundle $E\cap X_0'\to L'$, and thus isomorphic to $\P^1_k$.

It follows that, whenever $C$ is a non-empty intersection of at least two distinct components in $\cY_k$, then
 $C$ is rational.
  Furthermore, $\widetilde{C}^o=C^o$ because $C$ is always contained in a component of multiplicity one.
 Thus by Theorem \ref{thm:dual}, $\cX_{\pfi}$ is not stably rational, because the component $X'_0$ of $\cY_k$ is not stably rational.
\end{proof}
\begin{remark}\label{rem:quart-faithful}
With a little more work, one can show that the model $\cX$ is $\LL$-faithful: the double covers $\widetilde{V}_j^o$ and $\widetilde{E}^o_0$ are rational because
 $X_0$ has ordinary double points at the images of $V_j$ and the generic point of $L$. A direct computation now shows that
 $\Vol(\cX_{\pfi})\equiv [X_0]$ modulo $\LL$.
\end{remark}

\begin{theorem}[Colliot-Th\'el\`ene -- Pirutka, Theorem 1.17 in \cite{CTP-stab}; see also \cite{HT16}]\label{thm:quartics}
Very general smooth quartic threefolds over $k$ are not stably rational.
\end{theorem}
\begin{proof}
Proposition \ref{prop:quart-model} yields the existence of
 a stably irrational smooth quartic threefold $V$ over an algebraically closed extension $k'$ of $k$.
  Let $S$ be the open subscheme of $\P(H^0(\P^4_k,\mathcal{O}(4)))$ parameterizing smooth quartic threefolds.
 By Corollary \ref{cor:stabratlocus}, the subset of $S$ parameterizing geometrically stably rational smooth quartic threefolds is a countable union of closed subsets. Its complement is non-empty because it has a $k'$-point parameterizing the variety $V$.
\end{proof}


\appendix
\section{Existence of the motivic volume and motivic reduction morphisms}\label{sec:WF}
In this section, we give a detailed proof of the existence of the motivic volume and motivic reduction morphisms in the case where the residue field $k$ is not necessarily algebraically closed, thus completing the proofs of Theorem \ref{thm:vol} and Proposition \ref{prop:red}. Our main tool will be the Weak Factorization Theorem, which we will use in the form of Theorem \ref{thm:WF}.
  A similar strategy is used in \cite{KT} to prove the existence of their specialisation maps for birational types.

\subsection{Weak factorization}
   If $X$ is a smooth and proper $K$-scheme and $\cX$ is a snc-model of $X$, then an {\em elementary blow-up} of $\cX$ is a
  blow-up $\cX'\to \cX$ whose center is a connected regular closed subscheme $Z$ of $\cX_k$ that has strict normal crossings with the special fiber $\cX_{k}$. The strict normal crossings condition means that, Zariski-locally around every point of $Z$, we can find an effective divisor $D_0$ on $\cX$ such that $D=D_0+\cX_k$ has strict normal crossings and $Z$ is an intersection of irreducible components of  $D$. This property implies that $\cX'$ is again an snc-model of $X$.

 \begin{theorem}\label{thm:WF}
 Let $X$ be a smooth and proper $K$-scheme. Let $\mathcal{I}$ be an invariant of snc-models of $X$, that is,
 a map from the set of isomorphism classes of snc-models of $X$ to some set $A$. Assume that, for every snc-model $\cX$ of $X$ and
  every elementary blow-up $\cX'\to \cX$, we have $\mathcal{I}(\cX)=\mathcal{I}(\cX')$. Then the invariant $\mathcal{I}$ takes the same value on all snc-models of $X$.
 \end{theorem}
 \begin{proof}
 For every $R$-model $\cX$ of $X$, an {\em admissible blow-up} is a blow-up $\cX'\to \cX$ at a coherent ideal sheaf that contains $t^{n}$,  for some positive integer $n$; thus it induces an isomorphism on the generic fibers.  These are called $X$-admissible blow-ups in \cite{flattening}.
   A composition of admissible blow-ups is again an admissible blow-up, by  \cite[5.1.4]{flattening}.

 Let $\cX$ and $\cY$ be snc-models of $X$. We will prove that $\mathcal{I}(\cX)=\mathcal{I}(\cY)$.
  We define $\cZ$ to be the schematic closure of the diagonal of $X\times_K X$ in $\cX\times_R \cY$. By \cite[5.7.12]{flattening}, there exists
  an admissible blow-up $\cZ'\to \cZ$ such that the induced morphism $\cZ'\to \cX$ is also an admissible blow-up. Likewise, there exists  an admissible blow-up $\cZ''\to \cZ'$ such that $\cZ''\to \cY$ is an admissible blow-up. By Hironaka's resolution of singularities, we can find an admissible blow-up $\cW\to \cZ''$ such that $\cW$ is an snc-model of $X$. The morphisms $\cW\to \cX$ and $\cW\to \cY$ are still admissible blow-ups.
By the Weak Factorization Theorem for admissible blow-ups of quasi-excellent schemes of characteristic zero \cite[1.2.1]{AbTe}, we can now connect $\cX$ and $\cY$ by a chain of elementary blow-ups and blow-downs. Thus $\mathcal{I}(\cX)=\mathcal{I}(\cY)$.
 \end{proof}

 The main technical difficulty in the proofs of the existence of the motivic volume and motivic reduction maps is to organize the local computations of various blow-ups in an efficient way. A convenient tool for this purpose is the language of logarithmic geometry, which essentially allows us to reduce to the case of toric varieties, where combinatorial arguments can be used. This avoids tedious calculations in local coordinates (especially when dealing with the covers $\widetilde{E}_J^o$ in the formula for the motivic volume). The main technical results we will need are already present in \cite{BuNi}.

 For a brief introduction to logarithmic geometry, we refer to \cite{kato-log}. In combination with the overview of regular log schemes in Section 3 of \cite{BuNi}, this provides all the necessary background material for the arguments presented below. All log structures in this paper are defined with respect to the Zariski topology, as in \cite{kato}. Following the conventions in \cite{BuNi}, we will speak of regular log schemes and smooth morphisms of log schemes instead of log regular log schemes and log smooth morphisms. Log schemes will be denoted by symbols of the form $\loga{X}$; the underlying scheme of $\loga{X}$ is denoted by $\cX$.

\subsection{Smooth log schemes}
 Let $\cX$ be a normal flat $R$-scheme of finite type, and let $D$ be a reduced Weil divisor on $\cX$.
 The divisor $D$ gives rise to a logarithmic structure on $\cX$, called the {\em divisorial log structure} induced by $D$.
We denote the resulting log scheme by $\loga{X}$, and we call it the log scheme associated with the pair $(\cX,D)$. The sheaf of monoids defining the log structure on $\loga{X}$ is given by
$$\mathcal{M}_{\loga{X}}=j_*\mathcal{O}^{\times}_{\cX\setminus D}\cap \mathcal{O}_{\cX}$$ where $j$ is the open embedding of $\cX\setminus D$ into $\cX$. Thus $\mathcal{M}_{\loga{X}}$ is the sheaf of functions on $\cX$ that are invertible on the complement of $D$.

 In all the examples of interest in this paper, the log scheme $\loga{X}$ will be {\em regular} in the sense of \cite{kato}. This notion formalizes the geometric idea that $(\cX,D)$ is a strictly toroidal pair. The local toroidal structure is encoded in the {\em characteristic sheaf} of $\loga{X}$, which is defined as
 $\mathcal{C}_{\loga{X}}=\mathcal{M}_{\loga{X}}/\mathcal{O}_{\cX}^{\times}$. For every point $x$ of $\cX$, we call
 the stalk $\mathcal{C}_{\loga{X},x}$ the {\em characteristic monoid} of $\loga{X}$ at $x$. This is the monoid of effective Cartier divisors on $\Spec \mathcal{O}_{\cX,x}$ supported on $D\cap \Spec \mathcal{O}_{\cX,x}$. We can think of the dual monoid $$\mathcal{C}^{\vee}_{\loga{X},x}=\Hom(\mathcal{C}_{\loga{X},x},\N)$$ as the monoid of integral points on a rational polyhedral cone that describes the toroidal structure of the pair $(\cX,D)$ locally at $x$; see \cite[3.2]{kato}. The points $x$ where $\mathcal{C}_{\loga{X},x}$ is nonzero are precisely the points on the divisor $D$; we call $D$ the {\em logarithmic boundary} of the log scheme $\loga{X}$.

\begin{exam}
Assume that $\cX$ is regular.
 Then $\loga{X}$ is regular if and only if $D$ is a strict normal crossings divisor on $\cX$ \cite[3.3.1]{BuNi}. In that case, for every point $x$ of $\cX$, the characteristic monoid $\mathcal{C}_{\loga{X},x}$ is isomorphic to $\N^r$, where $r$ is the number of irreducible components of $D$ that pass through $x$.
\end{exam}

Let $D_i,\,i\in I$ be the irreducible components of $D$.
 A {\em logarithmic stratum} of $\loga{X}$ is a connected component of a set of the form
 $$\bigcap_{j\in J}D_j\setminus \left(\bigcup_{i\notin J}D_i\right) $$
 where $J$ is a subset of $I$. In particular, taking for $J$ the empty set, we see that the connected components of $\cX\setminus D$ are logarithmic strata. The characteristic sheaf $\mathcal{C}_{\loga{X}}$ is constant along each logarithmic stratum \cite[3.2.1]{BuNi}. The regularity of $\loga{X}$ implies that all the logarithmic strata are regular, if we endow them with their reduced induced structures. The {\em fan} $F(\loga{X})$ of $\loga{X}$ is the subspace of $\cX$ consisting of the generic points of the logarithmic strata, endowed with the pullback of the characteristic sheaf $\mathcal{C}_{\loga{X}}$. A dictionary with the classical notion of rational polyhedral fan is provided in \cite[9.5]{kato}.
 For every point $\sigma$ of $F(\loga{X})$, we denote the logarithmic stratum containing $\sigma$ by $E(\sigma)^o$. To every proper subdivision of the fan $F(\loga{X})$ as defined in \cite[9.7]{kato}, one can associate a proper birational morphism of regular log schemes $\loga{Y}\to \loga{X}$ as in \cite[9.9]{kato}, which we call a {\em toroidal modification} of $\loga{X}$. An important class of examples are normalized blow-ups of closures of logarithmic strata.

  We set $S=\Spec R$ and we denote by $S^\dagger$ the log scheme obtained by endowing $S$ with the divisorial log structure induced by its unique closed point. This is called the standard log structure on $S$. If $D$
   contains the reduced special fiber $\cX_{k,\mathrm{red}}$, then the structural morphism $\cX\to S$ induces a morphism of log schemes $\loga{X}\to S^{\dagger}$.
 By \cite[3.6.1]{BuNi}, regularity of $\loga{X}$ implies that the morphism of log schemes $\loga{X}\to S^{\dagger}$ is smooth, because $k$ has characteristic zero. Conversely, if
 $\loga{Y}$ is a fine and saturated log scheme and $\loga{Y}\to S^{\dagger}$ is a smooth morphism of log schemes, then $\loga{Y}$ is regular \cite[8.2]{kato}, and the log structure on $\loga{Y}$ is the divisorial log structure induced by some reduced Weil divisor on $\loga{Y}$ that contains $\cY_{k,\mathrm{red}}$ \cite[11.6]{kato}.

\subsection{Existence of the motivic volume}
We will prove a refinement of Theorem \ref{thm:vol} that also provides a formula for the motivic volume in terms of
smooth logarithmic models, generalizing the formula for snc-models.

Let $\loga{X}$ be a smooth fine and saturated log scheme of finite type over $S^{\dagger}$, and denote by $D$ its logarithmic boundary.
 Let $\sigma$ be a point in $F(\loga{X})\cap \cX_k$. We denote by $\overline{t}$ the residue class of $t$  in the characteristic monoid $\mathcal{C}_{\loga{X},\sigma}$.
 The {\em rank} of $\loga{X}$ at $\sigma$ is the Krull dimension of the monoid $\mathcal{C}_{\loga{X},\sigma}$, or, equivalently,  the rank of the abelian group $\mathcal{C}^{\mathrm{gp}}_{\loga{X},\sigma}$ associated with the monoid $\mathcal{C}_{\loga{X},\sigma}$.
 This rank is denoted by $r(\sigma)$. It is equal to the codimension of the stratum $E(\sigma)^o$ in $\cX$.
 We define the {\em horizontal rank} $r_h(\sigma)$ of $\loga{X}$ at $\sigma$ as the dimension of
 the face of the dual monoid $\mathcal{C}^{\vee}_{\loga{X},\sigma}$ defined by $\overline{t}=0$, and
 the {\em vertical rank} $r_v(\sigma)$ as the difference $r(\sigma)-r_h(\sigma)$.
  We have $r_v(\sigma)=r(\sigma)$ if and only if $D=\cX_{k,\mathrm{red}}$ locally around $\sigma$ (in the logarithmic literature, one then says that the log structure is {\em vertical} at $\sigma$).

 Let $\rho$ be the largest positive integer such that $\overline{t}$  is divisible by $\rho$ in $\mathcal{C}_{\loga{X},\sigma}$. This invariant is called the {\em root index} of $\loga{X}$ at $\sigma$ in \cite{BuNi}. If the underlying scheme $\cX$ of $\loga{X}$ is regular, then $\rho$ is the greatest common divisor of the multiplicities of the components of $\cX_k$ that pass through $\sigma$, by Example 4.1.1 in \cite{BuNi}.

 For every positive integer $m$, we set $S(m)=\Spec R(m)$, and we denote by $\cX(m)$ the normalization of $\cX\times_S S(m)$.
 We denote  by $\widetilde{E}(\sigma)^o$ the reduced inverse image of the log stratum $E(\sigma)^o$ in $\cX(\rho)$. It is equipped with a natural good action of the group scheme $\mu_{\rho}$, via the action of $\mu_{\rho}$ on $\cX(\rho)$. If $m$ is any positive multiple of
 $\rho$, then the natural morphism $\cX(m)\to \cX(\rho)$ induces an isomorphism
 $$(\widetilde{E}^o(\sigma)\times_{\cX(\rho)}\cX(m))_{\mathrm{red}}\to \widetilde{E}(\sigma)^o,$$
  by \cite[4.1.2]{BuNi}. Thus to compute $\widetilde{E}(\sigma)^o$, we may freely replace $\rho$ by any positive multiple.

  We define the motivic volume
  $\Vol(\loga{X})$ by
  $$\Vol(\loga{X})=\sum_{\sigma\in F(\loga{X})\cap \cX_k}(1-\LL)^{r_v(\sigma)-1}[\widetilde{E}(\sigma)^o]$$
 in $\gro^{\gal}(\Var_{k})$.
  When $\cX$ is a flat normal $R$-scheme of finite type such that the log scheme $\loga{X}$
associated with $(\cX,\cX_{k,\mathrm{red}})$ is regular,
 we will write $\Vol(\cX)$ instead of $\Vol(\loga{X})$.

\begin{lemma}\label{lemm:vollocal}
The motivic volume $\Vol(\loga{X})$ is local on $\loga{X}$, in the following sense: if
$\{\cU^{\dagger}_\alpha,\,\alpha\in A\}$ is a finite open cover of $\loga{X}$, then
 $$\Vol(\loga{X})=\sum_{\emptyset \neq B\subset A}(-1)^{|B|-1}\Vol(\cap_{\beta\in B}\cU^{\dagger}_\beta). $$
\end{lemma}
\begin{proof}
This follows immediately from the definition of the motivic volume and the scissor relations in the Grothendieck ring.
\end{proof}

\begin{lemma}\label{lemm:euler}
Let $\gamma$ be a rational polyhedral cone in $\R^d$, for some positive integer $d$, and let $\Sigma$ be the fan consisting of the faces of $\gamma$. Let $\delta$ be a face of $\gamma$, and let $\Sigma'$ be a proper rational polyhedral subdivision of $\Sigma$ that contains the face $\delta$.
 We denote by
$\gamma_1,\ldots,\gamma_n$ the cones in $\Sigma'$ that intersect the relative interior of $\gamma$.

 Let $\delta'$ be a face of $\delta$,
let $I$ be the set of indices $i$ in $\{1,\ldots,r\}$ such that $\gamma_i\cap \delta=\delta'$, and set $$A_{\delta,\delta'}=\sum_{i\in I}(-1)^{\dim(\gamma_i)}.$$ Then
$A_{\delta,\delta'}=(-1)^{\dim(\gamma)}$ if $\delta'=\delta$, and
$A_{\delta,\delta'}=0$ otherwise.
 \end{lemma}
\begin{proof}
 For every rational polyhedral cone $\varepsilon$, the value $(-1)^{\dim(\varepsilon)}$ is equal to the compactly supported Euler characteristic of the relative interior of $\varepsilon$.
 Since the union of the relative interiors of the cones $\gamma_i$ that contain $\delta$ is homeomorphic to $\gamma$, the equality $A_{\delta,\delta}=(-1)^{\dim(\gamma)}$ follows from the addivity of the compactly supported Euler characteristic. Thus we may assume that $\delta\neq \delta'$. Applying an inclusion-exclusion argument over all possible intersections of the cones $\gamma_i$ with the face $\delta$,
  we can write
 $$A_{\delta,\delta'}=\sum_{\varepsilon}(-1)^{\dim(\varepsilon)-\dim(\delta')}A_{\varepsilon,\varepsilon}=(-1)^{\dim(\gamma)}\sum_{\varepsilon}(-1)^{\dim(\varepsilon)-\dim(\delta')}$$
 where the sum is taken over all the faces $\varepsilon$ of $\delta$ that contain $\delta'$. The sum in the last member of this equality is equal to the compactly supported Euler characteristic
 of the projection $\overline{\delta}$ of $\delta$ onto the quotient of $\R^d$ by the real vector space space spanned by $\delta'$. This Euler characteristic vanishes because $\overline{\delta}$  is a strictly convex cone of positive dimension.
\end{proof}

 \begin{prop}\label{prop:modif}
The motivic volume $\Vol(\loga{X})$ is invariant under toroidal modifications of $\loga{X}$ that are isomorphisms on the generic fiber $\cX_K$.
  \end{prop}
\begin{proof}
Let $h\colon \loga{Y}\to \loga{X}$ be a toroidal modification, associated with a proper subdivision $h_F\colon F(\loga{Y})\to F(\loga{X})$ of the fan of $\loga{X}$.
 It suffices to show that, for every point $\sigma$ of $F(\loga{X})\cap \cX_k$, we have
 $$(1-\LL)^{r_v(\sigma)-1}[\widetilde{E}(\sigma)^o]=\sum_{\tau\in h_F^{-1}(\sigma)}(1-\LL)^{r_v(\tau)-1}[\widetilde{E}(\tau)^o]$$ in $\gro^{\gal}(\Var_{k})$.

  By \cite[4.1.3]{BuNi}, we have for every $\tau$ in $h_F^{-1}(\sigma)$ a natural morphism
  $\widetilde{E}(\tau)^o\to \widetilde{E}(\sigma)^o$, and $\widetilde{E}(\tau)^o$ is a $\mathbb{G}^{r(\sigma)-r(\tau)}_{m,k}$-torsor over $\widetilde{E}(\sigma)^o$. As observed in
  in the proof of \cite[7.2.1]{BuNi}, this torsor is $\gal$-equivariant if we put the trivial $\gal$-action on  $\mathbb{G}^{r(\sigma)-r(\tau)}_{m,k}$. By \cite[7.1.1]{BuNi}, this implies that
  $$[\widetilde{E}(\tau)^o]=(\LL-1)^{r(\sigma)-r(\tau)}[\widetilde{E}(\sigma)^o]$$
  in $\gro^{\gal}(\Var_{k})$. Thus it is enough to prove the following equalities: for every
  nonnegative integer $m$, we have
  $$\sum_{\tau\in h_F^{-1}(\sigma),\,r_h(\tau)=m}(-1)^{r(\tau)} = \left\{\begin{array}{ll} (-1)^{r(\sigma)}, & \mbox{if }m= r_h(\sigma)
  \\[1ex] 0, & \mbox{otherwise.}
  \end{array}\right.$$
   We write $\R_{\geq 0}$ for the additive monoid of nonnegative real numbers, and  $\sigma^{\vee}_{\R}$ for the strictly convex rational polyhedral cone
   $\Hom(\mathcal{C}_{\loga{X},\sigma},\R_{\geq 0})$ generated by the dual monoid $\mathcal{C}^{\vee}_{\loga{X},\sigma}$.
   For every $\tau$ in $h_F^{-1}(\sigma)$, the real cone $\tau^{\vee}_{\R}$ is defined analogously.
 By the dictionary provided in \cite[9.5]{kato}, we can
   view the cones $\tau^{\vee}_{\R}$ as the cones in a proper subdivision of   $\sigma^{\vee}_{\R}$ that intersect the relative interior of $\sigma^{\vee}_{\R}$. Since $h$ is an isomorphism on the generic fibers, this subdivision contains the face of $\sigma^{\vee}_{\R}$ defined by $\overline{t}=0$.
     Thus the result follows from Lemma \ref{lemm:euler}.
\end{proof}

 \begin{prop}\label{prop:horiz}
 The motivic volume $\Vol(\loga{X})$ is independent of the log structure:
 if $\loga{Y}$ is another fine and saturated log scheme that is smooth and of finite type over $S^{\dagger}$, and $h\colon \loga{Y}\to \loga{X}$ is a morphism of log schemes that is an isomorphism on the underlying schemes, then $\Vol(\loga{X})=\Vol(\loga{Y})$.
  \end{prop}
\begin{proof}
We will use the isomorphism $h$ to identify the schemes $\cX$ and $\cY$.
 We denote by $D_{\cX}$ and $D_{\cY}$ the logarithmic boundary divisors of $\cX$ and $\cY$.
 Then $D_{\cX}$ is contained in $D_{\cY}$, and $D_{\cX}$ contains $\cX_{k,\mathrm{red}}$.
 It follows that every logarithmic stratum of $\loga{X}$ is a union of logarithmic strata of $\loga{Y}$. Therefore, it is enough to show that, if $E(\tau)^o$ is a logarithmic stratum of
 $\loga{Y}$ contained in $\cY_k$, and $E(\sigma)^o$ is the unique logarithmic stratum of $\loga{X}$ that contains $E(\tau)^o$, then the vertical rank of $\loga{X}$ at $\sigma$ is equal to the vertical rank of $\loga{Y}$ at $\tau$ (we do not need to compare of the root indices, because we can freely replace them by a common multiple in the construction of the covers $\widetilde{E}(\sigma)^o$ and $\widetilde{E}(\tau)^o$).

  The morphism of characteristic monoids
 $\mathcal{C}_{\loga{X},\sigma}\to \mathcal{C}_{\loga{Y},\sigma}$
induced by $h$ is an isomorphism, because it is injective and the definition of regularity for log schemes implies that these monoids have the same dimension (namely, the codimension of the closure of $\{\sigma\}$ in $\cX\cong \cY$).
  The point $\tau$ lies in the closure of $\{\sigma\}$, so that we have a cospecialization morphism
 $\mathcal{C}_{\loga{Y},\tau}\to \mathcal{C}_{\loga{Y},\sigma}$. It follows from \cite[10.1]{kato} that the dual of this morphism  identifies
 $\mathcal{C}_{\loga{Y},\sigma}^{\vee}$ with a face of $\mathcal{C}_{\loga{Y},\tau}^{\vee}$.
 Every one-dimensional face of $\mathcal{C}_{\loga{Y},\tau}^{\vee}$ that is not included in
 $\mathcal{C}_{\loga{Y},\sigma}^{\vee}$ corresponds to a component of $D_{\cY}$ that is not included in $\cY_{k,\mathrm{red}}$. Thus $\overline{t}=0$ on all those one-dimensional faces.
 Therefore, the vertical ranks of $\sigma$ and $\tau$ are equal.
\end{proof}

\begin{prop}\label{prop:indepmod}
Let $X$ be a smooth and proper $K$-scheme. Let $\loga{X}$ be a fine and saturated log scheme, endowed with a smooth morphism of finite type $\loga{X}\to S^{\dagger}$, such that the underlying scheme $\cX$ is an $R$-model of $X$.  Then $\Vol(\loga{X})$ only depends on $X$, and not on $\loga{X}$ or $\cX$.
\end{prop}
\begin{proof}
By Proposition \ref{prop:modif} and toroidal resolution of singularities for regular log schemes \cite[10.4]{kato}, we may assume that $\cX$ is an snc-model of $X$. By Proposition \ref{prop:horiz}, we may further assume that the log structure on $\loga{X}$ is induced by the reduced special fiber $\cX_{k,\mathrm{red}}$.

 By Theorem \ref{thm:WF}, it suffices to show that, for every snc-model $\cX$ of $X$, the motivic volume $\Vol(\cX)$ is invariant under elementary blow-ups.
 Let $Z$ be a connected regular closed subscheme of $\cX_k$ that has strict normal crossings with $\cX_k$, and let $h\colon \cY\to \cX$ be the blow-up of $\cX$ at $Z$.
 We can compute the motivic volume $\Vol(\cX)$ locally on $\cX$, by Lemma \ref{lemm:vollocal}.   Thus we may assume that there exists a reduced divisor $D_0$ on $\cX$, flat over $R$, such that $D=\cX_{k}+D_0$ has strict normal crossings and $Z$ is an intersection
 of irreducible components of $D$. Let $\loga{X}$ be the log scheme associated with $(\cX,D)$.
Then $\Vol(\loga{X})=\Vol(\cX)$  by Proposition \ref{prop:horiz}, and $Z$ is the closure of a logarithmic stratum of $\loga{X}$.
 Moreover, if we set $D'=(h^{*}D)_{\mathrm{red}}$ and denote by $\loga{Y}$ the
 log scheme associated with $(\cY,D')$, then $\loga{Y}$ is regular and
 $\loga{Y}\to \loga{X}$ is a toroidal modification. Thus the result follows from Propositions \ref{prop:modif} and \ref{prop:horiz}.
\end{proof}

In view of Proposition \ref{prop:indepmod}, we can define the motivic volume of a smooth and proper $K$-scheme $X$ by the formula
$$\Vol_K(X)=\sum_{\sigma\in F(\loga{X})\cap \cX_k}(1-\LL)^{r_v(\sigma)-1}[\widetilde{E}(\sigma)^o]$$
 in $\gro^{\gal}(\Var_{k})$,
where $\loga{X}$ is any fine and saturated smooth log scheme over $S^{\dagger}$ whose underlying scheme is a model of $X$. If $\cX$ is an snc-model of $X$ and $\loga{X}$ is the log scheme associated with $(\cX,\cX_{k,\mathrm{red}})$, then our formula is equivalent to formula \eqref{eq:volX} in Theorem \ref{thm:vol}.
 We now prove that this invariant behaves well under blow-ups of smooth centers in $X$.

\begin{prop}\label{prop:blowup}
Let $X$ be a smooth and proper $K$-scheme, let $Z$ be a connected smooth closed subscheme of $X$, and let $Y\to X$ be the blow-up of $X$ along $Z$. We denote by $E$ the exceptional divisor of $Y$.
Then $\Vol_K(Y)-\Vol_K(E)=\Vol_K(X)-\Vol_K(Z)$ in $\gro^{\gal}(\Var_{k})$.
\end{prop}
\begin{proof}
By Hironaka's resolution of singularities, we can find an snc-model $\cX$ of $X$ such that the schematic closure $\cZ$ of $Z$ in $\cX$ has strict normal crossings with $\cX_k$. Then $\cZ$ is an snc-model for $Z$, the blow-up $\cY$ of $\cX$ along $\cZ$ is an snc-model for $Y$, and the inverse image $\cE$ of $\cZ$ in $\cY$ is an snc-model for $E$. It suffices to show that $\Vol(\cY)-\Vol(\cE)=\Vol(\cX)-\Vol(\cZ)$.

  We write $\loga{X}$, $\loga{Y}$, $\loga{Z}$ and $\loga{E}$ for the log schemes obtained by endowing $\cX$, $\cY$, $\cZ$ and $\cE$
  with the log structures induced by their reduced special fibers. All of these log schemes are smooth over $S^{\dagger}$.
  Moreover, at every point of $\cZ$, the morphism of log schemes $\loga{Z}\to \loga{X}$ induces an isomorphism of characteristic monoids at every point of $\loga{Z}$, and the analogous property holds for $\loga{E}\to \loga{Y}$. In other words, these closed immersions of log schemes are {\em strict}.
     Looking at the explicit formulas for
  motivic volumes of all these log schemes, and using the scissor relations in the Grothendieck ring, we see that it
  is enough to prove the following property: let $m$
be a positive integer, and denote by $\cX(m)$ the normalization of $\cX\times_S S(m)$. Then
$\cW=\cX(m)\times_{\cX}\cZ$ is the normalization of $\cZ\times_S S(m)$ (the corresponding result then also holds for $\cY$ and $\cE$).

   By \cite[3.7.1]{BuNi}, $\cX(m)$ is the underlying scheme of the fiber product $\cX(m)^{\dagger}=\loga{X}\times^{\mathrm{fs}}_{S^\dagger}S(m)^{\dagger}$ in the category of fine and saturated log schemes. Since $\loga{Z}\to \loga{X}$ is strict, the fiber product $\loga{W}=\cX(m)^{\dagger}\times_{\loga{X}}\loga{Z}$ in the category of log schemes is already fine and saturated, since it is a strict closed logarithmic subscheme of $\cX(m)^{\dagger}$.
   Thus $\loga{W}$ is isomorphic to the fine and saturated fiber product
   $$\cX(m)^{\dagger}\times^{\mathrm{fs}}_{\loga{X}}\loga{Z}\cong \loga{Z}\times^{\mathrm{fs}}_{S^\dagger}S(m)^{\dagger} $$
   which implies that the underlying scheme $\cW$ of $\loga{W}$ is the normalization of $\cZ\times_S S(m)$.
    \end{proof}

\begin{prop}\label{prop:mult}
If $X$ and $Y$ are smooth and proper $K$-schemes, then
$$\Vol_K(X\times_K Y)=\Vol_K(X)\cdot \Vol_K(Y)$$ in
 $\gro^{\gal}(\Var_{k})$.
\end{prop}
\begin{proof}
Let $\cX$ and $\cY$ be snc-models of $X$ and $Y$, respectively.
We write $Z=X\times_K Y$ and we denote by $\cZ$ the normalization of $\cX\times_R \cY$.
 Let $\loga{X}$, $\loga{Y}$ and $\loga{Z}$ be the log schemes associated with the pairs $(\cX,\cX_{k,\mathrm{red}})$, $(\cY,\cY_{k,\mathrm{red}})$ and  $(\cZ,\cZ_{k,\mathrm{red}})$.
 By the same argument as in the proof of \cite[3.7.1]{BuNi}, we can identify $\loga{Z}$ with the
 fiber product $\loga{X}\times^{\mathrm{fs}}_{S^\dagger}\loga{Y}$ in the category of fine and saturated log schemes. In particular, $\loga{Z}$ is fine and saturated, and it is smooth over $S^{\dagger}$ because $\loga{X}$ and $\loga{Y}$ are smooth over $S^{\dagger}$, and smoothness is preserved by fine and saturated base change and composition. Thus we can use the model $\loga{Z}$ to compute the motivic volume of $Z$. On any regular log scheme $\loga{W}$, the points of the fan $F(\loga{W})\subset \cW$ are precisely the points $w$ whose codimension in $\cW$ is equal to the dimension of the characteristic monoid $\mathcal{C}_{\loga{W},w}$. It follows that a point of $\cZ$ lies in the fan $F(\loga{Z})$ if and only if its images in $\cX$ and $\cY$ lie in the fans of $\loga{X}$ and $\loga{Y}$, respectively. Thus for every $\sigma$ in $F(\loga{X})$ and every $\sigma'$ in $F(\loga{Y})$, the preimage of $E(\sigma)^o\times_k E(\sigma')^o$ in $\cZ_k$ is a union of logarithmic strata.

Let $m$ be a positive integer that is divisible by all the root indices of $\loga{X}$, $\loga{Y}$ and $\loga{Z}$. Then, if $m$ is sufficiently divisible, the morphisms $\cX(m)^{\dagger}\to S(m)^{\dagger}$ and $\cY(m)^{\dagger}\to S(m)^{\dagger}$ are {\em saturated} in the sense of \cite[II.2.1]{tsuji}, by \cite[II.3.4]{tsuji} (in fact, one can deduce from \cite[II.4.2]{tsuji} that it suffices to take for $m$ the least common multiple of the multiplicities of the components in $\cX_k$ and $\cY_k$, but we don't need this here).
This implies that $\cZ(m)^{\dagger}$ is the fibered product $\cX(m)^{\dagger}\times_{S(m)^{\dagger}}\cY(m)^{\dagger}$ in the category of log schemes. In particular, the natural morphism of schemes $\cZ(m)\to \cX(m)\times_{S(m)}\cY(m)$ is an isomorphism, and it is $\mu_m$-equivariant if we endow  $\cX(m)\times_{S(m)}\cY(m)$ with the diagonal action.
 Thus for every point $\sigma$ in $F(\loga{X})$ and every point $\sigma'$ in $F(\loga{X})$, we have a $\mu_m$-equivariant isomorphism
 $$\bigsqcup_{\tau}\widetilde{E}(\tau)^o\to \widetilde{E}(\sigma)^o\times_k \widetilde{E}(\sigma')^o$$ where $\tau$ runs through all the points of $F(\loga{Z})$ in the preimage of $(\sigma,\sigma')$.
Since $r(\tau)=r(\sigma)+r(\sigma')$ and the horizontal ranks are zero,
 it follows that
 $$\Vol(\loga{Z})=\Vol(\loga{X})\cdot \Vol(\loga{Y}).$$
\end{proof}

 \begin{prop}\label{prop:basech}
Let $K(n)=k\llpar t^{1/n} \rrpar$ be a finite extension of $K$ in $\pfi$, and denote by $\gal(n)$ the kernel of the projection $\gal\to \mu_n$. Then for every smooth and proper $K$-scheme $X$,
the volume $\Vol_{K(n)}(X\times_K K(n))$ is the image of $\Vol_K(X)$ under the restriction morphism  $$\Res^{\gal}_{\gal(n)}\colon \gro^{\gal}(\Var_k)\to \gro^{\gal(n)}(\Var_k).$$
 \end{prop}
\begin{proof}
We choose an snc-model $\cX$ of $X$, and we denote by $\cX(n)$ the normalization of $\cX\times_S S(n)$. We write $\cX(n)^{\dagger}$ for the log scheme associated with $(\cX(n), \cX(n)_{k,\mathrm{red}})$.
 We have already recalled from \cite[3.7.1]{BuNi} that we can identify $\cX(n)^{\dagger}$ with the fiber product
 $\loga{X}\times^{\mathrm{fs}}_{S^{\dagger}}S(n)^{\dagger}$ in the category of fine and saturated log schemes.
In particular, $\cX(n)^{\dagger}$ is fine and saturated, and smooth over $S(n)^{\dagger}$, by the preservation of smoothness by fine and saturated base change.
 It also follows that, for every log stratum $E(\sigma)^o$ of $\loga{X}$, the reduced inverse image in $\cX(n)$ is a union of log strata $E(\tau)^o$ of $\cX(n)^{\dagger}$.
 Moreover, for each of these strata $E(\tau)^o$, we have
 $r_v(\sigma)=r(\sigma)$ and $r_v(\tau)=r(\tau)$ because the log structure on the generic fibers is trivial, and  $r(\sigma)=r(\tau)$ by the description of the characteristic monoids on a fine and saturated base change in \cite[\S3.7]{BuNi}.
 Finally, the sum of the classes $[\widetilde{E}(\tau)^o]$
in  $\gro^{\gal(n)}(\Var_k)$ equals
 $$\Res^{\gal}_{\gal(n)}([\widetilde{E}(\sigma)^o])$$ because we can compute the covers $\widetilde{E}(\sigma)^o$ and $\widetilde{E}(\tau)^o$ on a further normalized base change of $\cX(n)$,  by \cite[4.1.2]{BuNi}. The assertion now follows from the definition of the motivic volume.
\end{proof}

Now we are ready to prove our refinement of Theorem \ref{thm:vol}.

\begin{theorem}\label{thm:vol2}
There exists a unique ring morphism
$$\Vol_{K}:\gro(\Var_{K})\to \gro^{\gal}(\Var_k)$$ such that, for every smooth and proper $K$-scheme
$X$ and every fine and saturated smooth log scheme $\loga{X}$ over $S^{\dagger}$ whose underlying scheme is an $R$-model of $X$,
we have
$$\Vol_{K}(X)=\sum_{\sigma\in F(\loga{X})}(1-\LL)^{r_v(\sigma)}[\widetilde{E}(\sigma)^o].$$
 Moreover, for every finite extension $K(n)=k\llpar t^{1/n}\rrpar$ of $K$ in $\pfi$ and every $K$-scheme of finite type $Y$, the volume $\Vol_{K(n)}(Y\times_K K(n))$ is the image of $\Vol_{K}(Y)$ under the ring morphism
$$\Res^{\gal}_{\gal(n)}\colon \gro^{\gal}(\Var_k)\to \gro^{\gal(n)}(\Var_k)$$
where $\gal(n)$ is the kernel of the projection $\gal\to \mu_n$.
\end{theorem}
\begin{proof}
The uniqueness and existence of $\Vol_K$ as a group homomorphism follow from Bittner's presentation of the Grothendieck group (Theorem \ref{thm:bittner}); the fact that our definition for $\Vol_K$ satisfies the blowup relations was proven in Proposition \ref{prop:blowup}. Proposition \ref{prop:mult} implies that $\Vol_K$ is a ring morphism, and its behaviour under extensions of $K$ was described in Proposition \ref{prop:basech}.
\end{proof}

\subsection{Existence of the motivic reduction}
Finally, we prove a refined version of Proposition \ref{prop:red} on the existence of the motivic reduction.

\begin{prop}\label{prop:motred2}
There exists a unique group homomorphism
$$\red:\gro(\Var_{K})\to \gro(\Var_k)$$ such that, for every smooth and proper $K$-scheme $X$
and every fine and saturated smooth log scheme $\loga{X}$ over $S^{\dagger}$ whose underlying scheme is an $R$-model of $X$,
we have
$$\red(X)=\sum_{\sigma\in F(\loga{X})}(1-\LL)^{r_v(\sigma)}[E(\sigma)^o].$$
The morphism $\red$ is a morphism of $\Z[\LL]$-modules.
\end{prop}
\begin{proof}
The proof is similar to that of Theorem \ref{thm:vol2}. Replacing the covers  $\widetilde{E}^o(\cdot)$ by the log strata $E^o(\cdot)$ in the proofs of Propositions \ref{prop:modif},  \ref{prop:horiz}, \ref{prop:indepmod} and \ref{prop:blowup}, one sees that our formula for $\red$ does not depend on the choice of $\loga{X}$, and that it satisfies Bittner's blow-up relations.
 To see that it is a morphism of $\Z[\LL]$-modules, it suffices to observe that $\red(X\times_K \mathbb{P}^1_K)=\red(X) [\mathbb{P}^1_k]$ for every smooth and proper $K$-scheme $X$, because for every snc-model $\cX$ of $X$, the product $\cX\times_R \mathbb{P}^1_R$ is an snc-model for $X\times_K \mathbb{P}^1_K$. Subtracting $\red(X)$ from both sides of the equality, we find that
 $\red([X]\LL)=\red(X)\LL$.
\end{proof}

 Alternatively, one can also deduce Proposition \ref{prop:motred2} from Theorem \ref{thm:vol2} by taking the quotient of $\Vol_K$ modulo $\gal$, as in the proof of Proposition \ref{prop:red}.
 If $\cX$ is an snc-model for $X$ and $\loga{X}$ is the log scheme associated with $(\cX,\cX_{k,\mathrm{red}})$, then the formula for $\red(X)$ in Proposition \ref{prop:motred2} is equivalent to formula \eqref{eq:redX} in Proposition \ref{prop:red}.

\providecommand{\arxiv}[1]{{\tt{arXiv:#1}}}

\end{document}